\documentclass[11pt]{article}
\usepackage{amsmath, amsthm, amssymb}
\usepackage{float,epsfig}
\usepackage{color}
\usepackage{cancel}
\usepackage{graphicx}
\usepackage{rotating}
\usepackage{lscape}
\usepackage{caption}
\usepackage{subcaption}
\usepackage{tikz}
\usepackage{booktabs}
\usepackage{todonotes}
\usepackage{algpseudocode}
\usepackage{algorithm}
\usepackage{hyperref}
\usepackage{cleveref}
\usepackage{comment}
\usetikzlibrary{arrows}
\usetikzlibrary{patterns}

\DeclareMathOperator{\im}{Im}
\DeclareMathOperator{\Jac}{Jac}

\begin{document}

\newtheorem{theorem}{\bf Theorem}[section]
\newtheorem{proposition}[theorem]{\bf Proposition}
\newtheorem{corollary}[theorem]{\bf Corollary}
\newtheorem{lemma}[theorem]{\bf Lemma}

\theoremstyle{definition}
\newtheorem{definition}[theorem]{\bf Definition}
\newtheorem{example}[theorem]{\bf Example}
\newtheorem{exam}[theorem]{\bf Example}

\theoremstyle{remark}
\newtheorem{remark}[theorem]{\bf Remark}
\newtheorem{observation}[theorem]{\bf Observation}
\newcommand{\nrm}[1]{|\!|\!| {#1} |\!|\!|}

\newcommand{\ba}{\begin{array}}
\newcommand{\ea}{\end{array}}
\newcommand{\von}{\vskip 1ex}
\newcommand{\vone}{\vskip 2ex}
\newcommand{\vtwo}{\vskip 4ex}
\newcommand{\dm}[1]{ {\displaystyle{#1} } }

\newcommand\independent{\protect\mathpalette{\protect\independenT}{\perp}}
\def\independenT#1#2{\mathrel{\rlap{$#1#2$}\mkern2mu{#1#2}}}

\newcommand{\be}{\begin{equation}}
\newcommand{\ee}{\end{equation}}
\newcommand{\beano}{\begin{eqnarray*}}
\newcommand{\eeano}{\end{eqnarray*}}
\newcommand{\inp}[2]{\langle {#1} ,\,{#2} \rangle}
\def\bmatrix#1{\left[ \begin{matrix} #1 \end{matrix} \right]}
\def\dmatrix#1{\left| \begin{matrix} #1 \end{matrix} \right|}
\def \noin{\noindent}
\newcommand{\evenindex}{\Pi_e}


\def \R{{\mathbb R}}
\def \C{{\mathbb C}}
\def \K{{\mathbb K}}
\def \J{{\mathbb J}}
\def \Q{{\mathbb Q}}
\def \calL{\mathcal{L}}

\def \calB{{B}}
\def \calK{\mathcal{K}}
\def \calC{\mathcal{C}}
\def \calP{\mathcal{P}}
\def \calD{{D}}
\def \calV{{V}}
\def \calG{{G}}
\def \calN{\mathcal{N}}
\def \calT{\mathcal{T}}
\def \calH{\mathcal{H}}
\def \calM{\mathcal{M}}
\def \calI{\mathcal{I}}
\def \calE{{E}}
\def \calU{{U}}
\def \norm{\nrm{\cdot}\equiv \nrm{\cdot}}

\def \tr{\mathrm{Tr}}
\def \lam{\lambda}
\def \sig{\sigma}
\def \Sig{\Sigma}
\def \Lam{\Lambda}
\def \ep{\epsilon}
\def \sgn{\mathrm{sgn}}
\def \det{\mathrm{det}}

\algrenewcommand\algorithmicrequire{\textbf{Input:}}
\algrenewcommand\algorithmicensure{\textbf{Output:}}


\title{One-connection rule for structural equation models}

\author{ Bibhas Adhikari\thanks{Department of Mathematics, Indian Institute of Technology Kharagpur, bibhas@maths.iitkgp.ac.in}
  \and Elizabeth Gross\thanks{Department of Mathematics, University of Hawai`i at M\={a}noa, egross@hawaii.edu}
  \and Marc H\"ark\"onen\thanks{Max Planck Institute for Mathematics in the Sciences, harkonen@mis.mpg.de}
  \and Elias Tsigaridas\thanks{Inria Paris, elias.tsigaridas@inria.fr}
  }

\date{}

\maketitle

{\small \noin{\bf Abstract.} Linear structural equation models are multivariate statistical models encoded by mixed graphs.  In particular, the set of covariance matrices for distributions belonging to a linear structural equation model for a fixed mixed graph $G=(V, D,B)$ is parameterized by a rational function with parameters for each vertex and edge in $G$. This rational parametrization naturally allows for the study of these models from an algebraic and combinatorial point of view. Indeed, this point of view has led to a collection of results in the literature, mainly focusing on questions related to identifiability and determining relationships between covariances (i.e., finding polynomials in the Gaussian vanishing ideal).  So far, a large proportion of these results has focused on the case when $D$, the directed part of the mixed graph $G$, is acyclic. This is due to the fact that in the acyclic case, the parametrization becomes polynomial and there is a description of the entries of the covariance matrices in terms of a finite sum.  We move beyond the acyclic case and give a closed form expression for the entries of the covariance matrices in terms of the one-connections in a graph obtained from $D$ through some small operations. This closed form expression then allows us to show that if $G$ is simple, then the parametrization map is generically finite-to-one.  Finally, having a closed form expression for the covariance matrices allows for the development of an algorithm for systematically exploring possible polynomials in the Gaussian vanishing ideal.
}

\section{Introduction}

A structural equation model (SEM) is a  multivariate statistical model having a parametrization 
induced by a mixed graph $G$; that is a graph having both directed and bidirected edges.
Because of its
flexibility and its ability to model the effect of latent random
variables, it has a wide applicability to a variety of fields including ecology, psychology, and epidemiology. For an ecological example, in \cite{wolfe2017}, the authors use structural equation models to understand how native Hawaiian birds, such as the `i`iwi and `apapane arrange life cycle events around climatically-influenced food resources. In particular, 
they use \emph{linear structural equation models} in their analysis, that in turn employs linear equations for the description of the model and the entries of the corresponding covariance matrices are rational functions
in the parameters associated to $G$.
These linear models consist the focus of our study.

    A linear structural equation model is determined by a mixed graph
$G=(V, D, B)$, where $V$ is a vertex set of size $|V|=n$,
 $D$ is a set of directed
edges, and $B$ is a set of bidirected edges. Let $\R^D$ be the set of matrices $\Lam = (\lambda_{ij})\in\R^{n \times n}$
where $\lam_{ij}\neq 0$ if and only if $(i, j)\in\calD$ and let $\R_{reg}^\calD$ denote the set of matrices $\Lambda \in \R^D$ such that
 $I- \Lambda$   is invertible.  Furthermore, let $PD_\calV$
denote the cone of symmetric positive definite matrices of
order $n \times n$, and let 
$$PD(\calB)=\{\Omega\in PD_\calV : \omega_{ij}=0 \, \mbox{ if } \,i\neq j \, \mbox{ and } \, i\leftrightarrow j\notin \calB \}.$$
The vertices of the graph $G$ represent random variables $X_i$, where
$1 \leq i \leq n$, which we view as a vector ${\bf X} = (X_1, \ldots, X_n)^T$.  The linear structural equation model associated to $G$ is the family of all multivariate Gaussian distributions $\mathcal N(0, \Sigma)$ with 
a covariance matrix $\Sigma$ belonging to the image of the following parametrization map:
\begin{equation}
  \label{eq:phi-G}
  \begin{array}{lclllll}
\phi_\calG  :  & \R_{reg}^\calD \times PD(\calB) & \rightarrow & PD_\calV \\
     & (\Lam,\Omega) & \mapsto & (I-\Lam)^{-T}\Omega
(I-\Lam)^{-1} ,
  \end{array}
\end{equation}

\noindent   where $I$ denotes the identity matrix
of order $n\times n$.

Linear structural equation models have  been studied using a combination of algebraic and combinatorial techniques (see \cite{drton2018algebraic} for a thorough review).  These techniques have been particularly useful when addressing issues of parameter identification and establishing covariance matrix relationships. For parameter identification,  we are interested in cases where the map $\phi_G$ is globally injective (\emph{global identifiability}) or locally injective (\emph{local identifiability}).  For example, by pairing algebra and combinatorics, the authors of \cite{brito2002new} show that if $G= (V, D, B)$ is simple and $D$ is acyclic, then $\phi_G$ is generically injective. In \cite{tian2002general}, combinatorics and algebra are further applied, showing how the problem of identifiability becomes easier by studying  subgraphs of the original mixed graph.  While in \cite{foygel2012}, the authors develop the combinatorial \emph{half-trek criterion} for establishing generic global identifiability.

Whereas identifiability is useful for meaningful parameter inference, covariance relationships can be used to test model compatibility \cite{bollen2000tetrad, chen2014testable, drton2008moments}. Treating the entries of the covariance matrix $\Sigma =   (I - \Lambda)^{-T} \Omega (I- \Lambda)$ as indeterminates, the set $\mathcal I(G)$ of all polynomials in the entries $\Sigma$ with real coefficients that evaluate to zero for every covariance matrix in the image of $\phi_G$ is the \emph{Gaussian vanishing ideal of G}. The ideal $\mathcal I(G)$ contains  all polynomial covariance relationships. In addition to testing model compatibility, properties of the Gaussian vanishing ideal can be used to answer questions about the dimension of the model, singularities, and establishing model equivalence.   In the algebraic setting,  since the model is described as the image of a rational map, it is often helpful to consider the Zariski closure of the statistical model, which is an irreducible variety in the space of symmetric, $n \times n$, matrices (see, for example, \cite{sethBook}, for more more details). The Gaussian vanishing ideal of $G$ is the radical ideal corresponding to this irreducible variety. Combinatorial techniques have been used with much success in the study of Gaussian vanishing ideals. For  example, the purely graphical \emph{trek separation} \cite{sullivant2010trek} and \emph{restricted trek separation} \cite{drton2018nested} criteria give rise to certain polynomials in the vanishing ideal. This approach can yield at least a subideal of the Gaussian vanishing ideal even in cases where it is no feasible to compute the full vanishing ideal.

Both the identifiability problem and the covariance relationship problem can benefit from a description of the entries of $\Sigma$ in terms of $\lambda$ and $\omega$ parameters. When $D$ is acyclic, the map described in \eqref{eq:phi-G} is a polynomial map and the entries of $\Sigma =  \phi (I - \Lambda)^{-T} \Omega (I- \Lambda)$ can be described combinatorially in terms of a finite sum of trek monomials.  A \emph{trek} between vertices $i$ and $j$ in a mixed graph $G=(V, D, B)$ is a triple $(P_L, P_M , P_R)$ of paths where $P_L$ is a directed path of directed edges from $D$ with sink $i$, $P_R$ is a directed path of directed edges in $D$ with sink $j$, and $P_M$ is either empty if the source of $P_L$ is also the source of $P_R$, or a single bidirected edge from B connecting the source of $P_L$ to the source of $P_R$; a \emph{trek monomial} is a monomial in the $\lambda$s and $\omega$s associated with the paths $(P_L, P_R)$ and peak of the trek $(P_M)$.  When $D$ contains cycles the sum describing the entries of $\Sigma$ is infinite and, to our knowledge, up until now, a closed form expression for the entries of $\Sigma$ isn't known.

In this paper, for the main theorem (Theorem \ref{formula:oneconnectionrule}), we take an algebraic and combinatorial matrix theory  approach \cite{brualdi2008combinatorial} to  give a closed form rational expression for each of the entries of $\Sigma$.  We refer to this description of the entries of $\Sigma$ as the \emph{ one-connection rule} as a complement to the \emph{trek rule} (see e.g. \cite{sullivant2010trek}). In particular, we study the matrix in terms of
the $1$-\textit{connections} associated with the \textit{Coates
  digraph} representation of $I-\Lam$.  This results in an alternative
characterization associated with the graphs,
which is more compact
and results in more efficient implementations (Sec.~\ref{subsec:experiments}).  We then use this formula to: (i) Show that if $G=(V,D,B)$ is simple, then $\phi_G$ is generically finite-to-one (Theorem \ref{thm:ident}), and
 (ii) Give an algorithm to guide the discovery of polynomials in the Gaussian vanishing ideal by finding possible monomial supports (Algorithm \ref{alg:pruneSupport}).

The rest of the paper is organized as follows.  Section \ref{sec:prelim} reviews spanning subgraphs, linear graphs, the Coates formula for the determinant of a matrix,  and other  preliminaries from combinatorial matrix theory that we need for
our study. Section \ref{sec:Graphical_models} applies the tools reviewed in the preceding section to graphical models, giving a closed form expression for the covariance matrix $\Sigma$ and showing that for simple mixed graphs, including graphs with cycles, the parametrization $\phi_G$ is generically finite-to-one.  In Section ~\ref{sec:linear-1-connection} we present algorithms to symbolically compute the covariance matrix $\Sigma = \phi_{G}(\Lambda, \Omega)$ given a fixed graph using linear subgraphs and 1-connections, and we
evaluate the implementation of our algorithms on collections of graphs. Finally, in Section~\ref{sec:test-in-ideal} we present an approach
that gives candidates for the monomial support of homogeneous polynomials in the Gaussian vanishing ideal using our understanding of the covariance matrix $\Sigma$ in terms of 1-connections and linear subgraphs. 

\section{Preliminaries} \label{sec:prelim}

Throughout this manuscript we use $[n] = \{1,2, \dotsc, n \}$. Given a directed graph $D=(\calV, \calE)$ with possible self-loops, a \textit{spanning subgraph} of $D$ is a subgraph of $D$ whose vertex set is $V$.  A \textit{linear subgraph} of $D$ is a spanning subgraph of $D$ in which
each vertex has indegree $1$ and outdegree $1$. Thus, a linear subgraph is a spanning collection of pairwise vertex-disjoint cycles. Note a self-loop on a vertex contributes one to the indegree and outdegree of that vertex.  A graph and its linear subgraphs are pictured in Figure \ref{linearsg}. 

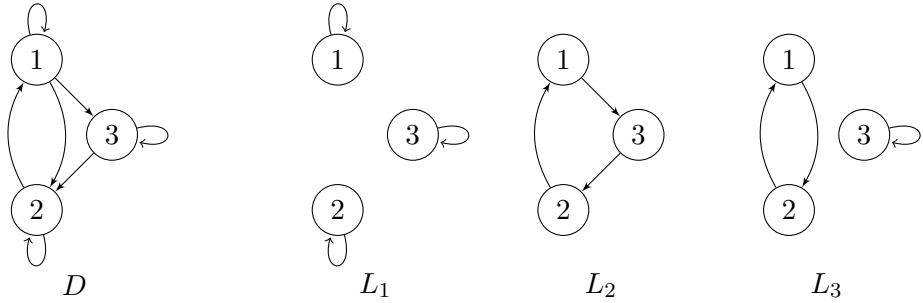
\begin{figure}
\centering
\begin{tikzpicture}
\tikzset{vertex/.style = {shape=circle,draw,minimum size=0.5em}}
\tikzset{edge/.style = {->,> = latex'}}
\node[vertex] (a) at (0,1) {1};
\node[vertex] (b) at (0,-1) {2};
\node[vertex] (c) at (1,0) {3};

\draw[edge] (a) to (c);
\draw[edge] (c) to (b);
\draw[edge] (b) to[bend left] (a);
\draw[edge] (a) to[bend left] (b);
\path (a) edge [loop above] node {} (a);
\path (b) edge [loop below] node {} (b);
\path (c) edge [loop right] node {} (c);
\node at (0.5,-2.0) {$D$};

\tikzset{vertex/.style = {shape=circle,draw,minimum size=0.5em}}
\tikzset{edge/.style = {->,> = latex'}}
\node[vertex] (a) at (4,1) {1};
\node[vertex] (b) at (4,-1) {2};
\node[vertex] (c) at (5,0) {3};

\path (a) edge [loop above] node {} (a);
\path (b) edge [loop below] node {} (b);
\path (c) edge [loop right] node {} (c);
\node at (4.5,-2.0) {$L_1$};

\tikzset{vertex/.style = {shape=circle,draw,minimum size=0.5em}}
\tikzset{edge/.style = {->,> = latex'}}
\node[vertex] (a) at (7,1) {1};
\node[vertex] (b) at (7,-1) {2};
\node[vertex] (c) at (8,0) {3};

\draw[edge] (a) to (c);
\draw[edge] (c) to (b);
\draw[edge] (b) to[bend left] (a);
\node at (7.5,-2.0) {$L_2$};

\tikzset{vertex/.style = {shape=circle,draw,minimum size=0.5em}}
\tikzset{edge/.style = {->,> = latex'}}
\node[vertex] (a) at (10,1) {1};
\node[vertex] (b) at (10,-1) {2};
\node[vertex] (c) at (11,0) {3};
\draw[edge] (b) to[bend left] (a);
\draw[edge] (a) to[bend left] (b);
\path (c) edge [loop right] node {} (c);
\node at (10.5,-2.0) {$L_3$};
\end{tikzpicture}
\caption{The directed graph $D$ and its linear subgraphs $L_1, L_2, L_3$}
\label{linearsg}
\end{figure}

Related to linear subgraphs of a directed graph are \emph{1-connections} (see e.g. \cite{brualdi2008combinatorial}), which are defined as follows.

\begin{definition}[$1$-connections of a directed graph]
Let $D=(V, E)$ be a directed graph. Let  $i,j \in \calV.$ An $1$-connection from $i$ to $j$ is a spanning subgraph  of $D$ with the following properties:
\begin{itemize}
    \item if $i\neq j$, then $i$ has indegree 0 and outdegree 1, $j$ has indegree 1 and outdegree 0, and every other vertex has indegree 1 and outdegree 1.
    
    \item if $i = j$, then $i = j$ has indegree 0 and outdegree 0, and every other vertex has indegree 1 and outdegree 1.
\end{itemize}

\end{definition}

\noindent In other words, an $1$-connection $C$ of $D=(V,E)$ from $i$ to $j$ is a spanning subgraph of $D$ that consists of a directed path $p$ from $i$ to $j$ (the path is of length zero if $i=j$) and a possibly empty collection of pairwise disjoint cycles that have no vertex in common with the path $p$. 

Note that, in general, for a pair $i,j$ there may be several possible 1-connections from $i$ to $j$. Given a directed graph $D$ and a pair of vertices $i, j$, we use $\calC_{i\to j}$ to denote the collection of all $1$-connections from $i$ to $j$. Some $1$-connections of the directed graph $D$ in \Cref{linearsg} are shown in \Cref{1csg}.

\begin{figure}
\centering
\begin{tikzpicture}
\tikzset{vertex/.style = {shape=circle,draw,minimum size=0.5em}}
\tikzset{edge/.style = {->,> = latex'}}
\node[vertex] (a) at (0,1) {1};
\node[vertex] (b) at (0,-1) {2};
\node[vertex] (c) at (1,0) {3};

\path (b) edge [loop below] node {} (b);
\path (c) edge [loop right] node {} (c);
\node at (0.5,-2.0) {
};

\tikzset{vertex/.style = {shape=circle,draw,minimum size=0.5em}}
\tikzset{edge/.style = {->,> = latex'}}
\node[vertex] (a) at (3,1) {1};
\node[vertex] (b) at (3,-1) {2};
\node[vertex] (c) at (4,0) {3};

\draw[edge] (a) to[bend left] (b);
\path (c) edge [loop right] node {} (c);
\node at (3.5,-2.0) {
};

\tikzset{vertex/.style = {shape=circle,draw,minimum size=0.5em}}
\tikzset{edge/.style = {->,> = latex'}}
\node[vertex] (a) at (6,1) {1};
\node[vertex] (b) at (6,-1) {2};
\node[vertex] (c) at (7,0) {3};

\draw[edge] (a) to (c);
\draw[edge] (c) to (b);
\node at (6.5,-2.0) {
};

\tikzset{vertex/.style = {shape=circle,draw,minimum size=0.5em}}
\tikzset{edge/.style = {->,> = latex'}}
\node[vertex] (a) at (9,1) {1};
\node[vertex] (b) at (9,-1) {2};
\node[vertex] (c) at (10,0) {3};
\draw[edge] (b) to[bend left] (a);
\draw[edge] (a) to[bend left] (b);
\node at (9.5,-2.0) {
};
\end{tikzpicture}
\caption{Some $1$-connections of the directed graph $D$ from \Cref{linearsg}.  The far left 1-connection belongs to the set $\mathcal C_{1 \to 1}$, the middle two belong to $\mathcal C_{1\to2}$, the far right 1-connection belongs to to $\mathcal C_{3 \to 3}$. }
\label{1csg}
\end{figure}
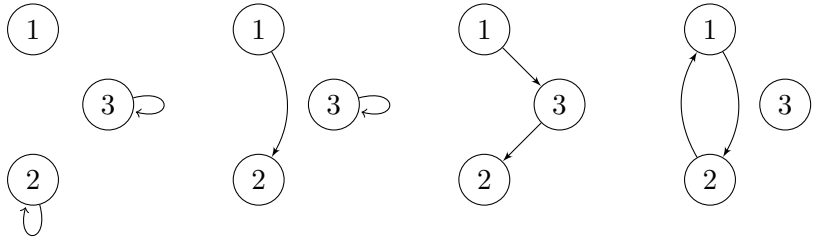

Given a directed graph $D$, we can obtain some 1-connections of  $D$ from linear subgraphs of $D$ in the following manner.   Let $L$ be a linear subgraph of $D$ that contains an edge from $j$ to $i$. Then if we delete this edge, we obtain a 1-connection from $i$ to $j$. If $i=j,$ the edge deleted is a self-loop.  Note that we can obtain some 1-connections $D$ this way, but not all.

In our study of structural equation models, we will consider directed graphs weighted by the entries of $\Lambda$.  We will denote weighted directed graphs as a triple $D=(V,E,W)$ where $W=[w_{ij}]$ is the matrix of edge weights.  Given a matrix $A$, we define its corresponding Coates digraph. 

\begin{definition}[Coates digraph]
Let $A=[a_{ij}]$ be a real square matrix of order $n.$ Then the Coates digraph $D_{A^T}$ corresponding to $A$ is defined as the weighted directed graph associated to $A^T$, that is, the graph $D_{A^T} = (V,E,A^{T})$ with vertex set $\calV =[n],$ edge set $\calE=\{(i, j) \,|\, a_{ji} \neq 0 \}$, and edge weight matrix $W=A^{T}=[a_{ji}]$.
\end{definition}

\begin{definition}
    Let $D=(V,E,W)$ be a weighted directed graph. The  weight product of $D$, denoted $w(D)$, is the product of  the weights on the edges in D, that is,
    \begin{align*}
        w(D) := \prod_{(i,j) \in E} w_{ij}.
    \end{align*}
    If $E = \emptyset$, then $w(D) := 1$.
    The  cycle number of $D$, denoted by $c(D)$, is the number of directed cycles (including self-loops) contained in $D$.
\end{definition}

Now we recall the Coates formula for the determinant of a square matrix $A$, which is written in terms of the weight products and cycle numbers of the linear subgraphs of the Coates digraph of $A$. 

\begin{definition}[Coates formula of determinant {\cite[Def 4.1.1]{brualdi2008combinatorial}}]\label{def:coatesDeterminant}
The determinant of a square matrix $A$ of order $n$ is given by 
\begin{align*}
  \det \ A = (-1)^n \sum_{L\in\calL} (-1)^{c(L)} w(L)=  \sum_{L\in\calL} (-1)^{n-c(L)} w(L),
\end{align*}
where $\calL$ is the set of linear subgraphs of the Coates digraph of $A$, i.e. $D_{A^T}$.
\end{definition}

Just as the determinant of a square matrix $A$ can be written in terms of the weight products and cycle numbers of the linear subgraphs of the Coates digraph of $A$, the entries of the inverse of $A$ can be written in terms of the weight products and cycle numbers of the linear subgraphs and $1$-connections of the Coates digraph of $A$.  In particular, $1$-connections play the role of cofactors for $A^{-1}$. 
\begin{theorem}[{\cite[Thm 5.3.2]{brualdi2008combinatorial}}]\label{formula:inv} 
  Let $A=[a_{ij}]$ be an invertible matrix. Then the $(j,i)$th entry of $A^{-1},$ say $a_{ji}'$, is given by
  \begin{align*}
      a_{ji}' = \frac{\sum_{C\in \calC_{i \to j}} (-1)^{c(C)+1} w(C)}{\sum_{L\in \mathcal{L}} (-1)^{c(L)} w(L)},
  \end{align*}
  where $\calC_{i \to j}$ and $\calL$ are respectively the set of 1-connections from $i$ to $j$ and the set of linear subgraphs of the Coates digraph  $D_{A^T}$ of $A$.
\end{theorem}

Both Definition \ref{def:coatesDeterminant} and Theorem \ref{formula:inv} will play key roles in the following section.

\section{The 1-connection rule and identifiability  for simple graphs}
\label{sec:Graphical_models}

While the previous section dealt exclusively with directed graphs, we now turn our attention back to mixed graphs. Let $G=(\calV, D, B)$ be a mixed graph, where $V$ is the set of vertices, $D$ is the set of directed edges, and $B$ is the set of bidirected edges. Recall that the structural equation model associated to $G$ is parametrized by two matrices $\Lambda \in \R_{\mathrm{reg}}^D$  and  $\Omega \in PD(B)$ and is the image of the map $\phi_G$ in equation \eqref{eq:phi-G} from $\R_{reg}^D \times PD(B)$ to  $PD_\calV$  where
$$\phi_{G} (\Lam, \Omega)=(I - \Lam)^{-T} \Omega (I - \Lam)^{-1}.$$

Now we will describe the entries of $\phi_{G} (\Lambda, \Omega)$ in terms of the combinatorics of $ G$. 
We construct a new weighted directed graph ${\widetilde{ D}}=(V, E(\widetilde{D}), I-\Lambda)$ by adding to $E(D)$ self-loops of weight $1$ to every vertex to obtain $E(\widetilde{D})$ and converting each edge weight $\lam_{ij}$ to $-\lam_{ij}$. Then ${\widetilde{D}}$ is the Coates digraph of $(I-\Lam)^{T}$, and, by \Cref{formula:inv}, we have the following proposition.
\begin{proposition}\label{prop:lamij}
  Let $G=(\calV, D, B)$ be a mixed graph on $n$ vertices, and $\Lam\in \R^\calD_{\mbox{reg}}.$ Then 
  \begin{align}\label{formula:mij}
    \left((I-\Lambda)^{-T}\right)_{ji}=\left((I-\Lambda)^{-1}\right)_{ij}=\frac{\sum_{C \in \calC_{i\to j}} (-1)^{c(C)+1} w(C)}{\sum_{L\in \mathcal{L}} (-1)^{c(L)} w(L)},
  \end{align}
  where $\calC_{i\to j}$ and $\calL$ are respectively the set of 1-connections from $i$ to $j$ and the set of linear subgraphs of ${\widetilde{D}}$.
\end{proposition}

\begin{proof}
  The proof follows from the construction of the directed graph $\widetilde{D}$ from $D$ and using Theorem \ref{formula:inv}.
\end{proof}

Note Proposition \ref{formula:mij} gives us a closed form expression for the entries of $(I-\Lam)^{-1}$ as rational functions as opposed to infinite series that we obtain in the series expansion $(I-\Lam)^{-1} = I + \Lam + \Lam^2 + \dotsb$. 

 By making several observations about acyclic graphs, we can see that Proposition \ref{prop:lamij} gives us the equation stated in Proposition 3.1 in \cite{sullivant2010trek} when $D$ is acyclic, that is, the $ij$th entry of $(I-\Lambda)^{-1}$ is the sum of path monomials over all paths from $i$ to $j$ in $D$.  First, observe that if $D$ is acyclic, the only linear subgraph of $\widetilde{D}$ is the subgraph containing every self-loop with weight $1$ and no other edges. Hence, 
\begin{align}\label{detcycle}
  \sum_{L\in \mathcal{L}} (-1)^{c(L)} w(L)=(-1)^n.
\end{align} 
Next, observe that when $D$ is acyclic, a one-connection $C \in \mathcal C_{i \to j}$ of $\widetilde D$ consists of a single path $p = ((i=i_0, i_1), (i_1, i_2), \hdots, (i_{l-1}, i_l=j))$ from $\widetilde D$ (and, consequently, $D$) and a collection of self-loops, thus, in this case, 
\begin{align*}
  w(C)=\prod_{(k,l) \in p} (-\lambda_{kl}),
\end{align*} 
when $i \neq j$.  When $i = j$, the path $p$ is necessarily the empty path, and we define  $w(C)=1$.

Finally, for any 1-connection $C \in \calC_{i \to j}$ we have $c(C) = n - (l+1)$, where $l$ is the length of the path from $i$ to $j$ in $C$. Therefore, for an acyclic graph $\calD$,
\begin{equation}\label{eq:1ml}
  \left((I-\Lambda)^{-1}\right)_{ij}= \frac{1}{(-1)^n}\sum_{p\in \mathcal{P}(i,j)} (-1)^{n-(l+1)+1}\prod_{(k,l) \in p} (-\lambda_{kl}) = \sum_{p\in \mathcal{P}(i,j)} \prod_{(k, l) \in p} \lam_{kl},
\end{equation} 
where $\mathcal P(i,j)$ is the set of paths from $i$ to $j$ in $D$  including the empty path when $i=j$. Thus, when $D$ is acyclic, equation \eqref{eq:1ml} is the equation stated in Proposition 3.1 in \cite{sullivant2010trek}. 

Now we provide a combinatorial meaning for the entries of the covariance matrix $\Sig$ corresponding to a mixed graph. Inspired by the term \textit{trek rule} used in \cite{sullivant2010trek} we call it the $1$\textit{-connection rule}.

\begin{theorem}[1-connection rule] \label{formula:phij}
Let $G=(V, D, B)$ be a mixed graph on $n$ vertices. Let $\Sigma=[\sig_{ij}]=(I-\Lam)^{-T}\Omega (I-\Lam)^{-1}\in \calM_\calG$ for some $\Lam\in \R^\calD_{\mbox{reg}}$ and $\Omega\in PD(\calB).$ Then 
  \begin{align} \label{formula:oneconnectionrule}
    \sig_{ij} = \frac{\sum_{k,l=1}^n \left[ \sum_{C\in \calC_{l\to i}} (-1)^{c(C)+1} w(C)\right] \omega_{lk} \left[ \sum_{C' \in \calC_{k\to j}} (-1)^{c(C')+1} w(C') \right]}{\left[\sum_{L\in \mathcal{L}} (-1)^{c(L)} w(L)\right]^2},
  \end{align}
  where $\calL$ is the set of linear subgraphs of $\widetilde{D}$, and $\calC_{a \to b}$ is the set of 1-connections of $\widetilde{D}$ from $a$ to $b$.
\end{theorem}
\begin{proof}The proof follows from the fact that $$\sig_{ij}= \sum_{l=1}^n \sum_{k=1}^n (I-\Lambda)^{-1}_{li} \omega_{lk} (I-\Lambda)^{-1}_{kj} $$ and equation (\ref{formula:mij}).
\end{proof}

Note that the formula of $\sig_{ij}$ reduces to the trek rule given in \cite{sullivant2010trek} when the mixed graph is acyclic due to equation (\ref{eq:1ml}). Recall that any trek $\tau$ between $i$ and $j$ is a path of the form
\begin{align*}
  \begin{cases}
    i = i_0 \leftarrow i_1 \leftarrow \dotsb \leftarrow i_s \leftrightarrow j_t \rightarrow \dotsb \rightarrow j_1 \rightarrow j_0 = j  , &\text{ if }i_s \neq j_t\\
    i = i_0 \leftarrow i_1 \leftarrow \dotsb \leftarrow i_s = j_t \rightarrow \dotsb \rightarrow j_1 \rightarrow j_0 = j  , &\text{ if }i_s = j_t
  \end{cases}
\end{align*}
We will denote by $w(\tau)$ the \emph{trek monomial} corresponding to $\tau$, given by
\begin{align*}
  w(\tau) = \prod_{k = 1}^s \lambda_{i_{k}, i_{k-1}} \cdot \omega_{i_s,j_t} \cdot \prod_{k=1}^t \lambda_{j_k,j_{k-1}}.
\end{align*}

\begin{corollary}[Trek rule, \cite{sullivant2010trek}]\label{cor:trekrule}
  Let $G = (V, D, B)$ be an acyclic mixed graph. Let $\Lambda \in \mathbb{R}^D$ and $\Omega \in PD(B)$. Then the entries of the covariance matrix $\Sigma=[\sig_{ij}]=(I-\Lam)^{-T}\Omega (I-\Lam)^{-1}$ are given by
  \begin{align*}
    \sigma_{ij} = \sum_{\tau \in \mathcal{T}(i,j)} w(\tau),
  \end{align*}
  where $\mathcal{T}(i,j)$ is the set of treks from $i$ to $j$.
\end{corollary}
\begin{proof}
  Since there are no directed cycles, the only linear subgraph of $\widetilde D$ is the graph consisting of only the vertices and self-loops. Thus the denominator in \eqref{formula:oneconnectionrule} will be equal to 1.

  Fix some $i,j$. Again because $D$ is acyclic, for any path $p$ from $a$ to $b$, there is only a single 1-connection with path $p$, namely the subgraph of $\widetilde D$ consisting of $p$ and self-loops on every vertex not present in $p$. This means $\sum_{C \in \calC_{l \to i}} (-1)^{c(C) + 1} w(C) = (-1)^{n}\sum_{p \in \calP(l,i)} \prod_{(r, s) \in p} \lam_{rs}$, where $\calP(l,i)$ is the set of directed paths from $l$ to $i$ in $D$. Thus the entry $\sigma_{ij}$ of the covariance matrix $\Sigma$ will be the sum
  \begin{align*}
    \sigma_{ij} = \sum_{\substack{k,l = 1,\dotsc,n\\ p \in \calP(l,i) \\ q\in \calP(k,j)}} \prod_{(r, s) \in p} \lam_{rs} \cdot \omega_{lk} \cdot \prod_{(t, u) \in q} \lam_{tu}.
  \end{align*}
  We observe that the monomial $\prod_{(r, s) \in p} \lam_{rs} \cdot \omega_{lk} \cdot \prod_{(t, u) \in q} \lam_{tu}$ is exactly the trek monomial corresponding to the trek between $i$ and $j$ given by the union of $p$ and $q$ and, if $l \neq k$, the bidirected edge $l \leftrightarrow k$.
\end{proof}

The trek rule from Corollary \ref{cor:trekrule} can be extended to the case where $D$ has cycles as in Proposition 2.2 in \cite{draisma2013positivity}, but then the sum becomes an infinite expression. Treating the infinite sum as a formal power series, a rational expression can be found for each $\sigma_{ij}$ on a case-by-case basis as illustrated in Example 4.2 in \cite{drton2018algebraic}.  The advantage of Theorem \ref{formula:phij} is that it gives a closed form formula for each $\sigma_{ij}$ directly as a rational expression.

We now illustrate Theorem \ref{formula:phij} with an example using
Example 4.2 from  \cite{drton2018algebraic}.

\begin{figure}
\centering
\begin{subfigure}[c]{0.80\textwidth}
 \resizebox{\linewidth}{!}{
\begin{tikzpicture}
\tikzset{vertex/.style = {shape=circle,draw,minimum size=0.5em}}
\tikzset{edge/.style = {<->,thick, red, > = latex'}}
\node[vertex] (a) at (0,1) {1};
\node[vertex] (b) at (0,-1.5) {2};
\node[vertex] (c) at (1.5,0) {3};
\node[vertex] (d) at (3.5,0) {4};

\draw[edge] (d) to[bend right] (c);
\path[->, thick, blue] (a) edge  node[pos=0.25,below left] {$\lam_{12}$} (b);
\path[->, thick, blue] (a) edge  node[pos=0.25,above right] {$\lam_{13}$} (c);
\path[->, thick, blue] (b) edge  node[pos=0.25,above right] {$\lam_{23}$} (c);
\path[->, thick, blue] (c) edge  node[pos=0.55,below left] {$\lam_{34}$} (d);
\path[->, thick, blue] (d) edge  node[pos=0.50,below right] {$\lam_{42}$} (b);
\node at (1.5,-2.0) {$G$};

\tikzset{vertex/.style = {shape=circle,draw,minimum size=0.5em}}
\tikzset{edge/.style = {<->, thick, red, > = latex'}}
\node[vertex] (a) at (6.5,1) {1};
\node[vertex] (b) at (6.5,-1.5) {2};
\node[vertex] (c) at (8.5,0) {3};
\node[vertex] (d) at (10.5,0) {4};

\path[->, thick, blue] (a) edge  node[pos=0.25,below left] {$-\lam_{12}$} (b);
\path[->, thick, blue] (a) edge  node[pos=0.25,above right] {$-\lam_{13}$} (c);
\path[->, thick, blue] (b) edge  node[pos=0.25,above right] {$-\lam_{23}$} (c);
\path[->, thick, blue] (c) edge  node[pos=0.25,above right] {$-\lam_{34}$} (d);
\path[->, thick, blue] (d) edge  node[pos=0.50,below right] {$-\lam_{42}$} (b);

\path (a) edge [loop above] node {1} (a);
\path (b) edge [loop below] node {1} (b);
\path (c) edge [loop above] node {1} (c);
\path (d) edge [loop above] node {1} (d);
\node at (7.5,-2.0) {$\widetilde{D}$};
\end{tikzpicture}}
\caption{The mixed graph $G=(V, D, B)$ and the directed graph $\widetilde{D}$. \\ $\ $ \\ $\ $ \\ $\ $}
\label{exap1}
\end{subfigure}

\begin{subfigure}[c]{0.6\textwidth}
\centering
 \resizebox{\linewidth}{!}{
\begin{tikzpicture}
\tikzset{vertex/.style = {shape=circle,draw,minimum size=0.5em}}
\tikzset{edge/.style = {<->, thick, red, > = latex'}}
\node[vertex] (a) at (0,1) {1};
\node[vertex] (b) at (0,-1) {2};
\node[vertex] (c) at (1.0,0) {3};
\node[vertex] (d) at (2.5,0) {4};
\path[->, thick, blue] (a) edge  node[pos=0.25,above right] {$-\lam_{13}$} (c);
\path[->, thick, blue] (c) edge  node[pos=-0.1,above right] {$-\lam_{34}$} (d);
\path[->, thick, blue] (d) edge  node[pos=0.50,below right] {$-\lam_{42}$} (b);
\node at (1.5,-2.0) {$\in \calC_{1\to 2}$};

\tikzset{vertex/.style = {shape=circle,draw,minimum size=0.5em}}
\tikzset{edge/.style = {<->, thick, red, > = latex'}}
\node[vertex] (a) at (6.0,1) {1};
\node[vertex] (b) at (6.0,-1) {2};
\node[vertex] (c) at (6.5,0) {3};
\node[vertex] (d) at (7.5,0) {4};

\path[->, thick, blue] (a) edge  node[pos=0.25,below left] {$-\lam_{12}$} (b);
\path (c) edge [loop above] node {1} (c);
\path (d) edge [loop above] node {1} (d);
\node at (6.5,-2.0) {$\in \calC_{1\to 2}$};
\end{tikzpicture}}
\caption{The $1$-connections in $\calC_{1 \rightarrow 2}$. \\ $\ $ \\ $\ $ \\ $\ $}
\label{exapC12}
\end{subfigure}

\begin{subfigure}[c]{0.60\textwidth}
\centering
 \resizebox{\linewidth}{!}{
\begin{tikzpicture}
\tikzset{vertex/.style = {shape=circle,draw,minimum size=0.5em}}
\tikzset{edge/.style = {<->, thick, red, > = latex'}}
\node[vertex] (a) at (7.0,1) {1};
\node[vertex] (b) at (7.0,-1) {2};
\node[vertex] (c) at (8.0,0) {3};
\node[vertex] (d) at (9.5,0) {4};

\path[->, thick, blue] (a) edge  node[pos=0.25,below left] {$-\lam_{12}$} (b);
\path[->, thick, blue] (b) edge  node[pos=0.2,right] {$-\lam_{23}$} (c);
\path[->, thick, blue] (c) edge  node[pos=-0.1,above right] {$-\lam_{34}$} (d);

\node at (8.5,-2.0) {$\in \calC_{1\to 4}$};

\tikzset{vertex/.style = {shape=circle,draw,minimum size=0.5em}}
\tikzset{edge/.style = {<->, thick, red, > = latex'}}
\node[vertex] (a) at (12.5,1) {1};
\node[vertex] (b) at (12.5,-1) {2};
\node[vertex] (c) at (13.5,0) {3};
\node[vertex] (d) at (15.0,0) {4};

\path[->, thick, blue] (a) edge  node[pos=0.25,above right] {$-\lam_{13}$} (c);
\path[->, thick, blue] (c) edge  node[pos=-0.1,above right] {$-\lam_{34}$} (d);

\path (b) edge [loop below] node {1} (b);
\node at (14.5,-2.0) {$\in \calC_{1\to 4}$};
\end{tikzpicture}}
\caption{The $1$-connections in $\calC_{1\rightarrow 4}$.\\ $\ $ \\ $\ $ \\ $\ $}
\label{exapC14}
\end{subfigure}

\begin{subfigure}[c]{0.60\textwidth}
\centering
 \resizebox{\linewidth}{!}{
\begin{tikzpicture}
\tikzset{vertex/.style = {shape=circle,draw,minimum size=0.5em}}
\tikzset{edge/.style = {<->, thick, red, > = latex'}}
\node[vertex] (a) at (0.0,1) {1};
\node[vertex] (b) at (0.0,-1) {2};
\node[vertex] (c) at (1.0,0) {3};
\node[vertex] (d) at (2.0,0) {4};

\path (a) edge [loop above] node {1} (a);
\path (c) edge [loop above] node {1} (c);
\path (d) edge [loop above] node {1} (d);
\node at (1.5,-2.0) {$\in \calC_{2\to 2}$};

\tikzset{vertex/.style = {shape=circle,draw,minimum size=0.5em}}
\tikzset{edge/.style = {<->, thick, red, > = latex'}}
\node[vertex] (a) at (5.0,1) {1};
\node[vertex] (b) at (5.0,-1) {2};
\node[vertex] (c) at (6.0,0) {3};
\node[vertex] (d) at (7.5,0) {4};

\path[->, thick, blue] (b) edge  node[pos=0.2,right] {$-\lam_{23}$} (c);
\path[->, thick, blue] (c) edge  node[pos=-0.1,above right] {$-\lam_{34}$} (d);

\path (a) edge [loop above] node {1} (a);
\node at (6.5,-2.0) {$\in \calC_{2\to 4}$};
\end{tikzpicture}}
\caption{The $1$-connections in $\calC_{2\to 2}$,  $\calC_{2 \to 4}$. \\ $\ $ \\ $\ $ \\ $\ $ }
\label{exapC22C24}
\end{subfigure}
\caption{A mixed graph $G=(V,D,B)$ and all the $1$-connections and linear subgraphs of  $\widetilde{D}.$}
\end{figure}

\begin{figure}\ContinuedFloat 
\centering
\begin{subfigure}[c]{0.60\textwidth}
\centering
 \resizebox{\linewidth}{!}{
\begin{tikzpicture}
\tikzset{vertex/.style = {shape=circle,draw,minimum size=0.5em}}
\tikzset{edge/.style = {<->, thick, red, > = latex'}}
\node[vertex] (a) at (6.5,1) {1};
\node[vertex] (b) at (6.5,-1) {2};
\node[vertex] (c) at (7.0,0) {3};
\node[vertex] (d) at (8.5,0) {4};

\path (a) edge [loop above] node {1} (a);
\path[->, thick, blue] (c) edge  node[pos=-0.1,above right] {$-\lam_{34}$} (d);
\path[->, thick, blue] (d) edge  node[pos=0.50,below right] {$-\lam_{42}$} (b);
\node at (7.5,-2.0) {$\in \calC_{3\to 2}$};

\tikzset{vertex/.style = {shape=circle,draw,minimum size=0.5em}}
\tikzset{edge/.style = {<->, thick, red, > = latex'}}
\node[vertex] (a) at (11.5,1) {1};
\node[vertex] (b) at (11.5,-1) {2};
\node[vertex] (c) at (12.0,0) {3};
\node[vertex] (d) at (13.5,0) {4};

\path[->, thick, blue] (c) edge  node[pos=-0.1,above right] {$-\lam_{34}$} (d);

\path (a) edge [loop above] node {1} (a);
\path (b) edge [loop below] node {1} (b);
\node at (13.0,-2.0) {$\in \calC_{3\to 4}$};
\end{tikzpicture}}
\caption{The $1$-connections in $\calC_{3\to 2}$, and $\calC_{3 \to 4}$. \\ $\ $ \\ $\ $ \\ $\ $}
\label{exapC32C34}
\end{subfigure}

\begin{subfigure}[c]{0.60\textwidth}
\centering
 \resizebox{\linewidth}{!}{
\begin{tikzpicture}
\tikzset{vertex/.style = {shape=circle,draw,minimum size=0.5em}}
\tikzset{edge/.style = {<->, thick, red, > = latex'}}
\node[vertex] (a) at (0.0,1) {1};
\node[vertex] (b) at (0.0,-1) {2};
\node[vertex] (c) at (0.5,0) {3};
\node[vertex] (d) at (1.5,0) {4};

\path[->, thick, blue] (d) edge  node[pos=0.50,below right] {$-\lam_{42}$} (b);

\path (a) edge [loop above] node {1} (a);
\path (c) edge [loop above] node {1} (c);
\node at (1.0,-2.0) {$\in \calC_{4\to 2}$};

\tikzset{vertex/.style = {shape=circle,draw,minimum size=0.5em}}
\tikzset{edge/.style = {<->, thick, red, > = latex'}}
\node[vertex] (a) at (5.0,1) {1};
\node[vertex] (b) at (5.0,-1) {2};
\node[vertex] (c) at (6.0,0) {3};
\node[vertex] (d) at (7.0,0) {4};


\path (a) edge [loop above] node {1} (a);
\path (b) edge [loop below] node {1} (b);
\path (c) edge [loop above] node {1} (c);
\node at (6.5,-2.0) {$\in \calC_{4\to 4}$};
\end{tikzpicture}}
\caption{The 1-connections in $\calC_{4\to 2}, \calC_{4\to 4}$. \\ $\ $ \\ $\ $ \\ $\ $}
\label{exapC41C44}
\end{subfigure}

\begin{subfigure}[c]{0.60\textwidth}
\centering
 \resizebox{\linewidth}{!}{
\begin{tikzpicture}
\tikzset{vertex/.style = {shape=circle,draw,minimum size=0.5em}}
\tikzset{edge/.style = {<->, thick, red, > = latex'}}
\node[vertex] (a) at (6.5,1) {1};
\node[vertex] (b) at (6.5,-1.0) {2};
\node[vertex] (c) at (7.0,0) {3};
\node[vertex] (d) at (8.0,0) {4};

\path (a) edge [loop above] node {1} (a);
\path (b) edge [loop below] node {1} (b);
\path (c) edge [loop above] node {1} (c);
\path (d) edge [loop above] node {1} (d);
\node at (7.5,-2.0) {$L_1$};

\tikzset{vertex/.style = {shape=circle,draw,minimum size=0.5em}}
\tikzset{edge/.style = {<->, thick, red, > = latex'}}
\node[vertex] (a) at (11.5,1) {1};
\node[vertex] (b) at (11.5,-1.0) {2};
\node[vertex] (c) at (12.5,0) {3};
\node[vertex] (d) at (14.0,0) {4};

\path[->, thick, blue] (b) edge  node[pos=0.8,left] {$-\lam_{23}$} (c);
\path[->, thick, blue] (c) edge  node[pos=-0.1,above right] {$-\lam_{34}$} (d);
\path[->, thick, blue] (d) edge  node[pos=0.50,below right] {$-\lam_{42}$} (b);

\path (a) edge [loop above] node {1} (a);
\node at (12.5,-2.0) {$L_2$};
\end{tikzpicture}}
\caption{The linear subgraphs of $\widetilde{D}$. \\ $\ $ \\ $\ $ \\ $\ $}
\label{exapL}
\end{subfigure}
\caption{A mixed graph $G=(V,D,B)$ and all the $1$-connections and linear subgraphs of  $\widetilde{D} $ (cont.)}
\label{exap:1}
\end{figure}
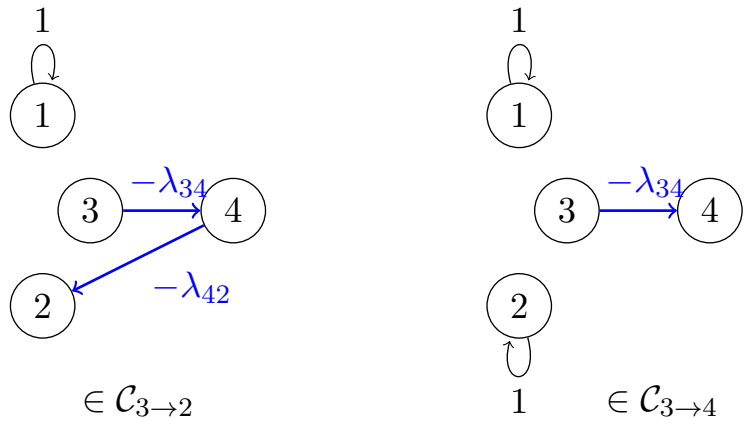
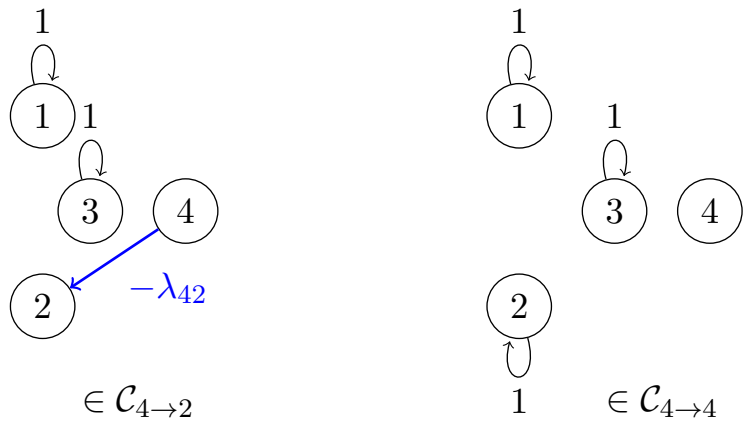
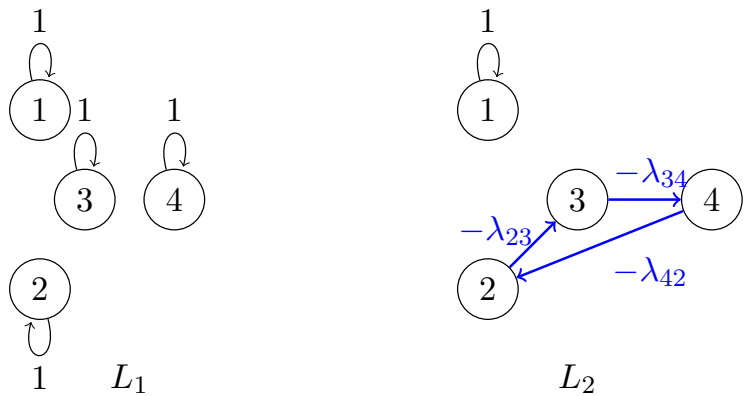

\begin{example}
Consider the mixed graph in Figure \ref{exap1}. Let us calculate $\sig_{24}.$ From the description of the linear subdigraphs of $\widetilde{D}$ in Figure \ref{exapL} the denominator of $\sig_{ij}$ from Theorem \ref{formula:inv} is given by $$\left[ (-1)^{c(L_1)} w(L_1) + (-1)^{c(L_2)} w(L_2)\right]^2 = ( 1 - \lam_{23}\lam_{34}\lam_{42} )^2.$$

We will look at the terms in the numerator that contain $\omega_{11}$, that is,
\begin{align*}
  \left( \sum_{C \in \calC_{1\to 2}}(-1)^{c(C)+1} w(C) \right) \omega_{11} \left(\sum_{C' \in \calC_{1\to 4}} (-1)^{c(C')+1} w(C')\right).
\end{align*}
Since $\calC_{1 \to 2}$ contains two 1-connections, we get the sum
\begin{align*}
  \left( \sum_{C \in \calC_{1\to 2}}(-1)^{c(C)+1} w(C) \right) =& (-1)^{0+1}(-\lambda_{13})(-\lambda_{34})(-\lambda_{42}) + (-1)^{2+1}(-\lambda_{12})\\ 
  =& \lambda_{13}\lambda_{34}\lambda_{42} + \lambda_{12}.
\end{align*}

\noindent The collection $\calC_{1 \to 4}$ also contains two 1-connections, so we have
\begin{align*}
  \left(\sum_{D\in \calC_{1\to 4}} (-1)^{c(D)+1} w(D)\right) =& (-1)^{0+1}(-\lambda_{12})(-\lambda_{23})(-\lambda_{34}) + (-1)^{1+1}(-\lambda_{13})(-\lambda_{34})\\
   =& \lambda_{12}\lambda_{23}\lambda_{34} + \lambda_{13}\lambda_{34}.
\end{align*}
Thus the terms  of the sum that contain $\omega_{11}$ will be $(\lambda_{13}\lambda_{34}\lambda_{42} + \lambda_{12}) \omega_{11} (\lambda_{12}\lambda_{23}\lambda_{34} + \lambda_{13}\lambda_{34})$. We perform  similar computations to obtain the full numerator of $\sigma_{24}$

Thus the terms  of the sum that contain $\omega_{11}$ will be $(\lambda_{13}\lambda_{34}\lambda_{42} + \lambda_{12}) \omega_{11} (\lambda_{12}\lambda_{23}\lambda_{34} + \lambda_{13}\lambda_{34})$. We perform  similar computations to obtain the full numerator of $\sigma_{24}$
\begin{align*} 
&(\lambda_{13}\lambda_{34}\lambda_{42} + \lambda_{12})\omega_{11} (\lam_{12}\lam_{23}\lam_{34} + \lam_{13}\lam_{34}) \\
+& \omega_{22}\lam_{23}\lam_{34} + \omega_{33}\lam_{34}^2\lam_{42} + \omega_{44}\lam_{42} + 2\omega_{34}\lam_{34}\lam_{42}.
\end{align*}

\end{example}

One of the benefits of Theorem \ref{formula:phij} is that it gives us a tool for establishing identifiability results.  In particular, our second main theorem below, Theorem \ref{thm:ident}, uses the closed form formula stated in Theorem \ref{formula:phij} to show that if $G$ is a simple mixed graph, then $\phi_G$ is generically finite-to-one, in other words, the covariance matrix $\Sigma$ is \emph{generically locally identifiable}.  Recall that $\phi_G$ is generically finite-to-one if $\phi_G$ is locally injective for almost all $(\Lambda, \Omega) \in \mathbb R_{reg}^D \times PD(B)$. A mixed graph $G=(V, D,B)$ is simple if there is at most one edge (directed or bidirected) between any two vertices $i \neq j$. 
\begin{theorem} \label{thm:ident}
If $G = (V, D, B)$ is simple, then $\phi_G$ is generically finite-to-one. 
\end{theorem}

\begin{proof}  The map $\phi_{G}$ is a rational map such that each coordinate function is of the form $(\phi_G)_{ij} = \sigma_{ij} = f_{ij}/g_{ij}$ where $f_{ij}$ and $g_{ij}$ are respectively the numerator and denominator of equation \eqref{formula:oneconnectionrule}. Note that the denominator in \eqref{formula:oneconnectionrule} is the same for all $i,j$, thus we can write $g_{ij}=g$ and $(\phi_G)_{ij} = \sigma_{ij} = f_{ij}/g$ for all $i,j$.  To show that $\phi_G$ is generically finite-to-one, we will show that $\dim (\im( \phi_G)) = |V| + |D| + |B|$.

Consider the  polynomial map $\tilde \phi_{G}$ defined by setting each coordinate function to be $(\tilde \phi_G)_{ij} = f_{ij}$.  Since (i) $\im (\tilde \phi_G) \subseteq \im (\phi_G)$, (ii) the Zariski closure of both $\im (\tilde \phi_G)$ and $\im (\phi_G)$ are irreducible, and (iii) the maximum dimension of $\im (\phi_G)$ is $|V| + |D| + |B|$, to show that $\dim (\im( \phi_G)) = |V| + |D| + |B|$, it suffices to show that $\dim (\im (\tilde \phi_G)) = |V| + |D| + |B|$.  We will do this showing the Jacobian matrix of $\tilde \phi_G$, which we will denote by $\Jac(\tilde \phi_G)$, has full generic rank.

The matrix $\Jac(\tilde \phi_G)$ is a $(|V| + |B| + |D|) \times {\binom{|V|}{2}} $ matrix where the $ij$th column is the gradient of $\sigma_{ij}$.  Let us look at the  $(|V| + |B| + |D|) \times (|V| + |B| + |D|)$ submatrix $M$ of $\Jac(\tilde \phi_G)$ that involves the columns corresponding to $\sigma_{ii}$ for all $i$ and $\sigma_{ij}$ if $(i,j)$ is an edge in $D$ or $B$.  We can label the rows of the matrix by the $\omega$ and $\lambda$ parameters.  Now let us consider the columns of this submatrix. 

Let $\mathcal C_{i \to j}^{kl}$ denote the set of one connections of $\tilde D$ from $i \to j$ that contain the edge $(k, l)$, and let $\mathcal C_{i \to j}^{-kl}$ denote the set of one connections of $\tilde D$ from $i \to j$ that do not contain the edge $(k, l)$. Note that the columns of $M$ corresponding to $\sigma_{ii}$ and $\sigma_{ij}$ have the following respective forms:

{\small
$$\nabla \sigma_{ii} = \begin{pmatrix}
\vdots \\
\frac{\partial \sigma_{ii}}{\partial \omega_{rr}}  \\
\vdots \\
\frac{\partial \sigma_{ii}}{\partial \omega_{rs}}  \\
\vdots \\
\frac{\partial \sigma_{ii}}{\partial \lambda_{rs}} \\
\vdots
\end{pmatrix} =  
\begin{pmatrix}
\vdots \\
 \left[ \sum_{C\in \calC_{r\to i}} (-1)^{c(C)+1} w(C)\right]  \left[\sum_{C' \in \calC_{r\to i}} (-1)^{c(C')+1} w(C') \right]  \\
\vdots \\
 \left[ \sum_{C\in \calC_{r\to i}} (-1)^{c(C)+1} w(C)\right]  \ \left[\sum_{C' \in \calC_{s\to i}} (-1)^{c(C')+1} w(C') \right]   \\
\vdots \\
\sum_{k,l=1}^n \left\{  \left[ \sum_{C\in \calC^{rs}_{l\to i}} (-1)^{c(C)+1}\frac{w(C)}{\lambda_{rs}} \right] \omega_{lk} \left[ \sum_{C' \in \calC^{- rs}_{k\to i}} (-1)^{c(C')+1} w(C') \right] \right.
\\
+  \left[ \sum_{C\in \calC^{-rs}_{l\to i}} (-1)^{c(C)+1}w(C) \right] \omega_{lk} \left[ \sum_{C' \in \calC^{rs}_{k\to i}} (-1)^{c(C')+1} \frac{w(C')}{\lambda_{rs}} \right]
\\
+  \left.\left[ 2\sum_{C\in \calC^{rs}_{l\to i}} (-1)^{c(C)+1}\frac{w(C)}{\lambda_{rs}} \right] \omega_{lk} \left[ \sum_{C' \in \calC^{rs}_{k\to i}} (-1)^{c(C')+1} w(C') \right] \right\}
\\
\vdots
\end{pmatrix},$$

{\small
$$\nabla \sigma_{ij} = \begin{pmatrix}
\vdots \\
\frac{\partial \sigma_{ij}}{\partial \omega_{rr}}  \\
\vdots \\
\frac{\partial \sigma_{ij}}{\partial \omega_{rs}}  \\
\vdots \\
\frac{\partial \sigma_{ij}}{\partial \lambda_{rs}} \\
\vdots
\end{pmatrix} =  
\begin{pmatrix}
\vdots \\
 \left[ \sum_{C\in \calC_{r\to i}} (-1)^{c(C)+1} w(C)\right]  \left[\sum_{C' \in \calC_{r\to j}} (-1)^{c(C')+1} w(C') \right]  \\
\vdots \\
 \left[ \sum_{C\in \calC_{r\to i}} (-1)^{c(C)+1} w(C)\right]  \ \left[\sum_{C' \in \calC_{s\to j}} (-1)^{c(C')+1} w(C') \right]   \\
\vdots \\
\sum_{k,l=1}^n \left\{  \left[ \sum_{C\in \calC^{rs}_{l\to i}} (-1)^{c(C)+1}\frac{w(C)}{\lambda_{rs}} \right] \omega_{lk} \left[ \sum_{C' \in \calC^{- rs}_{k\to j}} (-1)^{c(C')+1} w(C') \right] \right.
\\
+  \left[ \sum_{C\in \calC^{-rs}_{l\to i}} (-1)^{c(C)+1}w(C) \right] \omega_{lk} \left[ \sum_{C' \in \calC^{rs}_{k\to j}} (-1)^{c(C')+1} \frac{w(C')}{\lambda_{rs}} \right]
\\
+  \left.2 \left[ \sum_{C\in \calC^{rs}_{l\to i}} (-1)^{c(C)+1}\frac{w(C)}{\lambda_{rs}} \right] \omega_{lk} \left[ \sum_{C' \in \calC^{rs}_{k\to j}} (-1)^{c(C')+1} w(C') \right] \right\}
\\
\vdots
\end{pmatrix}$$
}

\noindent With these columns in mind, let us evaluate $M$ at the point $p= (\ldots, \omega_{ii}, \ldots, \omega_{ij}, \ldots, \lambda_{ij}, \ldots)$ where $\omega_{ij} = 1$ for all $\omega_{ij}$ and $\lambda_{ij}=0$ for all $\lambda_{ij}$ . By noting that the edge weight $w(C) =1$ only when $C$ is a 1-connection from $i \to i$ with all other vertices covered by self-loops and recalling that $G$ is simple,  we see: $\frac{\partial \sigma_{ii}}{\partial \omega_{rr}} |_{p} = 1$ when $r=i$ and zero otherwise; $\frac{\partial \sigma_{ii}}{\partial \omega_{rs}} |_{p} = 0$ when $r \neq s$; $\frac{\partial \sigma_{ii}}{\partial \lambda_{rs}} |_{p} = 0$ when $(r, s) \notin B$; $\frac{\partial \sigma_{ij}}{\partial \omega_{rr}} |_{p} = 0$ when $i \neq j$; $\frac{\partial \sigma_{ij}}{\partial \omega_{rs}} |_{p} = 1$ when $r =i$ and $s = j$, and zero otherwise; and 
$\frac{\partial \sigma_{ij}}{\partial \lambda_{rs}} |_{p} = 1$ when $r =i$ and $s = j$, and zero otherwise. Thus, up to a possible reordering of columns, $M$ evaluated at $p$ is the identity matrix and, consequently, $\Jac(\tilde \phi_G)$, has full rank when evaluated at $p$.  Therefore $\Jac(\tilde \phi_G)$ has full generic rank and $\phi_G$ is generically finite-to-one.
}
\end{proof}

\section{Symbolic computation of covariance matrices using linear subgraphs and 1-connections}
\label{sec:linear-1-connection}

In order to explore theoretical properties of structural equation models, such as identifiability, having a straightforward way to compute the covariance matrix $\Sigma = \phi_{G}(\Lambda, \Omega) = (I-\Lambda)^{-T}\Omega(I-\Lambda)^{-1}$ for matrices $\Lambda$ and $\Omega$ with undetermined entries   can be quite helpful.  As the number of nodes in the mixed graph gets large, the naive symbolic computation of the covariance matrix $\Sigma = \phi_{G}(\Lambda, \Omega) = (I-\Lambda)^{-T}\Omega(I-\Lambda)^{-1}$ via Gaussian elimination becomes computationally challenging.
Ideally, we want to exploit the sparsity structure of the graph to compute the covariance matrix. 

In the acyclic case, the trek rule achieves this: we can construct all treks by considering every path between every pair of
vertices  and gluing them together along bidirected edges or along vertices when the paths start at the same vertex. In the cyclic
case however, the trek rule gives us entries of the covariance matrix
as a formal power series. The advantage of the representation of
$(I-\Lambda)^{-1}$ in Proposition \ref{prop:lamij} is that entries of the inverse are written explicitly as  rational functions, with sums of a
finite number of terms in the numerator and denominator. The
computational problem now becomes to find all linear subgraphs and
1-connections.
We will observe in Section~\ref{subsec:experiments}
that in the case of random graphs with few cycles, using the
1-connection method to symbolically compute the covariance matrix is
much faster than using ``naive'' symbolic matrix inversion. Throughout
this section we will assume that $\calG = (\calV, \calD, \calB)$ is a
mixed graph with $n$ vertices so that $I, \Lambda, \Omega$ and
$\Sigma$ are all $(n\times n)$ matrices.

In the next subsection, Section ~\ref{sec:linear-subgraphs}, we present
the algorithm for the computation of the linear subgraphs, next in
Section~\ref{sec:1-connections} we present the algorithm for the
1-connections.  Section~\ref{sec:main-alg} introduces our main
algorithm and in Section~\ref{subsec:experiments} we present our
implementation and experiments on various data sets.

\subsection{Linear subgraphs}
\label{sec:linear-subgraphs}

We begin by computing the determinant of $(I-\Lambda)^{T}$ by using
linear subgraphs. Recall the Coates formula for the determinant in Definition 
\ref{def:coatesDeterminant}, that is,
\begin{align*}
    \det A = \sum_{L \in \calL} (-1)^{n-c(L)} w(L),
\end{align*}
where $\calL$ is the set of linear subgraphs of $A^T$. As noted in
Section \ref{sec:prelim}, a linear subgraph is a spanning collection of
vertex-disjoint cycles in the Coates digraph, which for the case of
$(I-\Lambda)^T$ is the graph $\widetilde \calD$, i.e. the directed part of $\calD$, with self-loops added to each vertex and negative weights. Hence any vertex-disjoint set of
cycles $S$ in $\calD$ gives a unique linear subgraph of $\widetilde D$ by adding
the self-loops of every vertex not present in $S$. Since the weights of
the self-loops are all 1, we have $w(L) = \prod (-\lambda_{i,j})$,
where the product is taken over all edges in $S$. If $S$ contains
$c_S$ cycles and $v_S$ vertices, we have to add $n - v_S$ self-loops to
create the linear subgraph $L \subseteq \widetilde D$. Thus the exponent of $(-1)$ becomes
\begin{align*}
    n - c(L) = n - (c_S + n - v_S) = v_S - c_S.
\end{align*}
The procedure above is summarized in Algorithm \ref{alg:linSubgraphDet}.

\begin{algorithm}
    \caption{Compute the determinant of $(I - \Lambda)$ using linear subgraphs}
    \label{alg:linSubgraphDet}
    \begin{algorithmic}[1]
        \Require A mixed graph $\calG = (\calV,\calD,\calB)$
        \Ensure A symbolic expression of $\mathrm{Det} = \det(I-\Lambda)$
        \Statex
        \Procedure{Det}{$\calG$}
            \State $\mathcal{C} \gets $\Call{Cycles}{$\calD$}
            \State $\mathcal{S} \gets $ all pairwise vertex disjoint subsets of $\mathcal{C}$
            \State $\mathrm{Det} \gets 0$
            \ForAll{$S \in \mathcal{S}$}
                \State $E_S \gets $\Call{Edges}{$S$}
                \State $v_S \gets \#$\Call{Vertices}{$S$}
                \State $c_S \gets \#$\Call{Cycles}{$S$}
                \State $W \gets \prod_{(i,j) \in E_S} (-\lambda_{i,j})$
                \State $\mathrm{Det} \gets \mathrm{Det} + (-1)^{v_S - c_S} W$
            \EndFor
            \State \Return $\mathrm{Det}$
        \EndProcedure
    \end{algorithmic}
\end{algorithm}

\subsection{1-connections}
\label{sec:1-connections}

Again, we will focus on computing the entries of $(I-\Lambda)^{-T}$, whose Coates digraph is $\widetilde \calD$. Given an invertible matrix $A$, we recall that its adjugate matrix $M$ is the numerator of the expression in Theorem \ref{formula:inv}, that is
\begin{align}\label{eq:oneconnmatrix}
    M_{ij} := \sum_{C \in \calC_{i \to j}} (-1)^{c(C) + 1} w(C),
\end{align}
where $\calC_{i\to j}$ is the collection of 1-connections from $i$ to $j$ in the Coates digraph of $A$.

We will compute the adjugate matrix of $A = (I - \Lambda)^T$. Recall that a 1-connection from $i$ to $j$ consists of a path from $i$ to $j$, and a set of cycles such that every vertex appears either in the path or cycles exactly once. In this case, a 1-connection $C$ from $i$ to $j$ consists of the following data: (i) a path $p$, (ii) a set $S$ consisting of $c_S$ cycles such that $\{p\} \cup S$ are pairwise vertex-disjoint, and (iii) self-loops for any vertex not appearing in either $p$ or $S$. Then
\begin{align*}
    (-1)^{c(C)+1} w(C) = (-1)^{c_S +n - v_S - v_p + 1} \prod \lambda_{k,l},
\end{align*}
where the product runs over all edges $k \to l$ appearing in $S \cup \{p\}$.
\Cref{alg:oneConnMatrix} summarizes the computation of the adjugate matrix.

\begin{algorithm}
    \caption{Compute the adjugate matrix of $(I-\Lambda)^T$.}
    \label{alg:oneConnMatrix}
    \begin{algorithmic}[1]
        \Require A mixed graph $\calG = (\calV,\calD,\calB)$
        \Ensure The symbolic $(n \times n)$ matrix $M = \det(I-\Lambda) \cdot (I-\Lambda)^{-T}$
        \Statex
        \Procedure{Adj}{$\calG$}
            \State $\mathcal{C} \gets $\Call{Cycles}{$\calD$}
            \State $\mathcal{S} \gets $ all pairwise vertex disjoint subsets of $\mathcal{C}$
            \State $M \gets 0$
            \ForAll{$1 \leq i,j \leq n$}
                \State $P \gets $ all directed paths from $i$ to $j$
                \ForAll{ $p \in P$}
                    \State $v_p \gets \#$\Call{Vertices}{$p$}
                    \State $\mathcal{S}_p \gets \{S \in \mathcal{S} \colon \text{$p$ and $S$ are pairwise vertex disjoint}\}$
                    \ForAll{$S \in \mathcal{S}_p$}
                        \State $E \gets$\Call{Edges}{$S \cup \{p\}$}
                        \State $v_S \gets \#$\Call{Vertices}{$S$}
                        \State $c_S \gets \#$\Call{Cycles}{$S$} 
                        \State $W \gets \prod_{(i,j) \in E} (-\lambda_{i,j})$
                        \State $M_{i,j} \gets M_{i,j} + (-1)^{c+n-v_S-v_p + 1} W$
                   \EndFor
                \EndFor
            \EndFor
            \State \Return M
        \EndProcedure
    \end{algorithmic}
\end{algorithm}

\subsection{Main algorithm}
\label{sec:main-alg}

\Cref{alg:mainAlgorithm} combines the subprocedures in  Algorithm \ref{alg:linSubgraphDet} and Algorithm  \ref{alg:oneConnMatrix}   to construct the covariance matrix $\Sigma =[\sigma_{ij}]=(I-\Lambda)^{-T}\Omega(I-\Lambda)^{-1}$ using only the combinatorics of the mixed graph $\calG$. The correctness of the algorithm is a result of Theorem \ref{formula:phij}. The denominator of $\sigma_{ij}$ is
\begin{align*}
    \left[\sum_{L \in \mathcal{L}} (-1)^{c(L)}w(L)\right]^2,
\end{align*}
which is equal to $(\det (I-\Lambda))^2$ by Definition \ref{def:coatesDeterminant}. For the numerator, using the adjugate matrix in \eqref{eq:oneconnmatrix}, we obtain $\sum_{k,l} M_{i,l} \Omega_{l,k} M_{j,k} = (M\Omega M^T)_{i,j}$.

\begin{algorithm}[h!]
    \caption{Compute the covariance matrix combinatorially}
    \label{alg:mainAlgorithm}
    \begin{algorithmic}[1]
        \Require A mixed graph $\calG = (\calV,\calD,\calB)$
        \Ensure The symbolic covariance matrix $\Sigma = (I-\Lambda)^{-T}\Omega(I-\Lambda)^{-1}$
        \Statex
        \Procedure{CovarianceMatrix}{$\calG$}
            \State $\Omega_{i,j} \gets \begin{cases} \omega_{i,j} & \text{if } i = j \text{ or } (i,j) \in \mathcal{B}\\
            0 & \text{otherwise}
            \end{cases}$
            \State $M \gets $\Call{Adj}{$\calG$}
            \State $D \gets$ \Call{Det}{$\calG$}
            \State \Return $(1/D^2) M \Omega M^T$
        \EndProcedure
    \end{algorithmic}
\end{algorithm}

\subsection{Computational experiments}
\label{subsec:experiments}

In this section we will compare algorithms for symbolically computing the covariance matrix $\phi_G(\Lambda,\Omega)$. We will compare the 1-connection method described in Algorithm \ref{alg:mainAlgorithm} against the ``naive'' method of computing $(I-\Lambda)^{-T}\Omega(I-\Lambda)^{-1}$ by symbolically inverting the matrix $(I - \Lambda)$ using the \texttt{Inverse} function in Mathematica. The dataset used for the computational experiments can be found in

\begin{quote}
\centering  
\url{https://gitlab.mis.mpg.de/harkonen/one-connections}
\end{quote}

\subsubsection{Cycle chains} 
\label{ssub:cycle_chains}
As a first experiment, we will consider chains of cycles. More precisely, we will construct a graph $\calG_{d,2\ell}$ consisting of $d$ directed cycles $c_0,\dotsc,c_{d-1}$ with $2\ell$ vertices in each. We will denote the vertices and edges in $c_i$ as $c_i = \{v_{i,0} \to v_{i,1} \to \dotsb \to v_{i,2\ell-1} \to v_{i,0}\}$. In addition, we connect each cycle with an edge $v_{i,\ell} \to v_{i+1,0}$ for all $i = 0,\dotsc,d-2$. As an example, \Cref{fig:cycleChainExample} depicts the graph $\calG_{3,4}$.

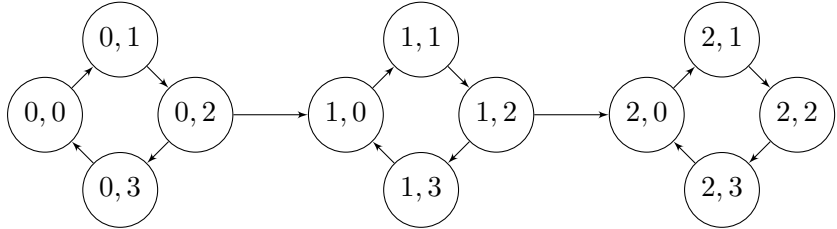
\begin{figure}
\centering
\begin{tikzpicture}
\tikzset{vertex/.style = {shape=circle,draw,minimum size=0.5em}}
\tikzset{edge/.style = {->,> = latex'}}
\node[vertex] (a) at (-5,0) {${0,0}$};
\node[vertex] (b) at (-4,1) {${0,1}$};
\node[vertex] (c) at (-3,0) {${0,2}$};
\node[vertex] (d) at (-4,-1) {${0,3}$};

\node[vertex] (e) at (-1,0) {${1,0}$};
\node[vertex] (f) at (0,1) {${1,1}$};
\node[vertex] (g) at (1,0) {${1,2}$};
\node[vertex] (h) at (0,-1) {${1,3}$};

\node[vertex] (i) at (3,0) {${2,0}$};
\node[vertex] (j) at (4,1) {${2,1}$};
\node[vertex] (k) at (5,0) {${2,2}$};
\node[vertex] (l) at (4,-1) {${2,3}$};

\draw[edge] (a) to (b);
\draw[edge] (b) to (c);
\draw[edge] (c) to (d);
\draw[edge] (d) to (a);

\draw[edge] (c) to (e);

\draw[edge] (e) to (f);
\draw[edge] (f) to (g);
\draw[edge] (g) to (h);
\draw[edge] (h) to (e);

\draw[edge] (g) to (i);

\draw[edge] (i) to (j);
\draw[edge] (j) to (k);
\draw[edge] (k) to (l);
\draw[edge] (l) to (i);
\end{tikzpicture}
\caption{An example of a chain of cycles, with $d = 3$ cycles of length $2\ell = 4$ each.}
\label{fig:cycleChainExample}
\end{figure}

The timings of the computations of the covariance matrices of $\calG_{d,2\ell}$ are shown in \Cref{fig:cycleChainPlots}, where the circles are timings, and the lines are fitted exponential functions of the form $\mathrm{Time} = ab^d$ for some $a,b$. When $2\ell=2$, the graphs have few vertices (ranging from 2 to 20), and the naive symbolic inversion of the matrix $I - \Lambda$ is faster than the 1-connection method. However, as the number of vertices increases, we observe that the 1-connection method eventually overtakes the naive method around $2\ell=6$.

\begin{figure}
  \centering
  \begin{subfigure}[t]{.5\textwidth}
    \centering
    \includegraphics[width=\textwidth]{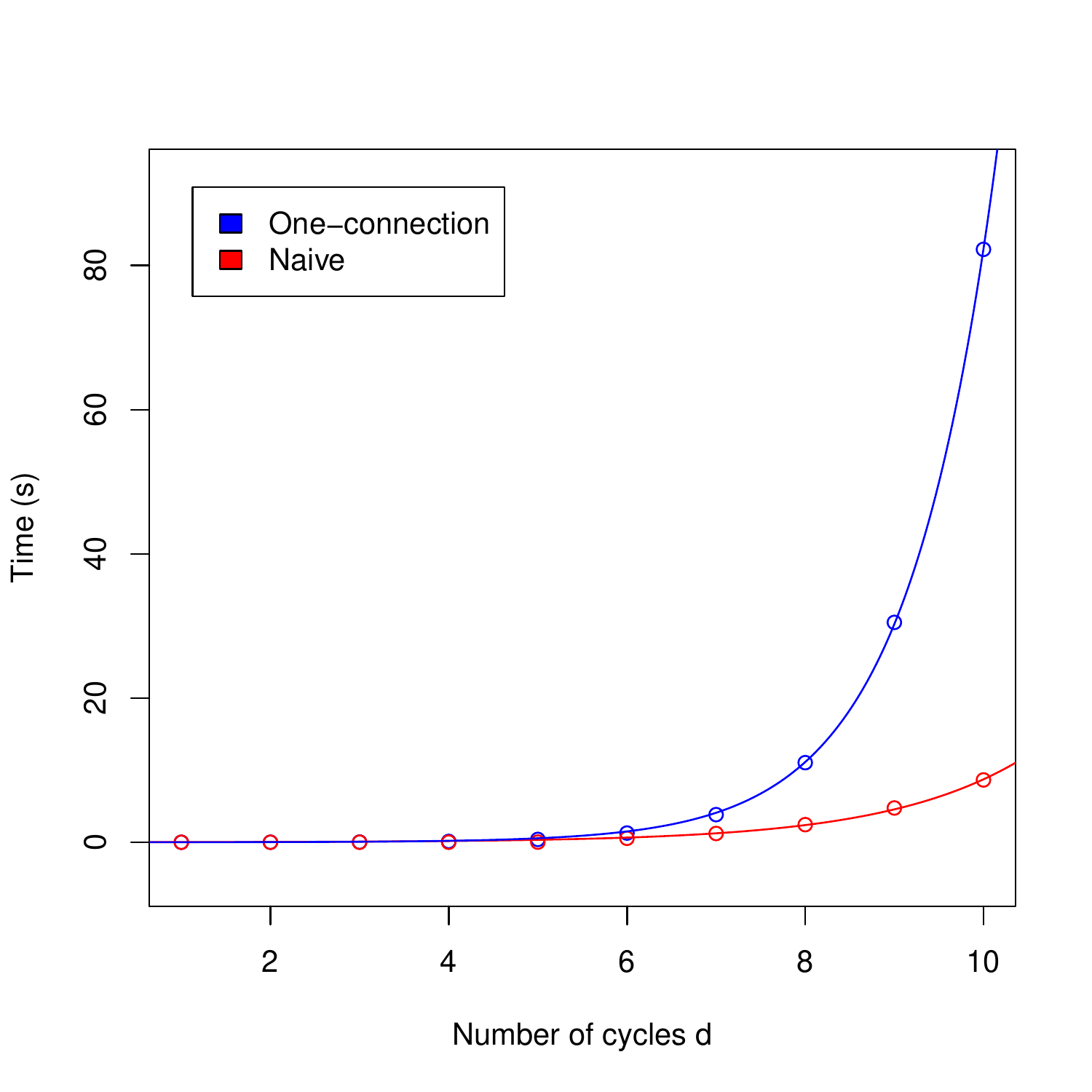}
    \caption{$2\ell=2$}
    \label{subfig:cycChain2}
  \end{subfigure}~
  \begin{subfigure}[t]{.5\textwidth}
    \centering
    \includegraphics[width=\textwidth]{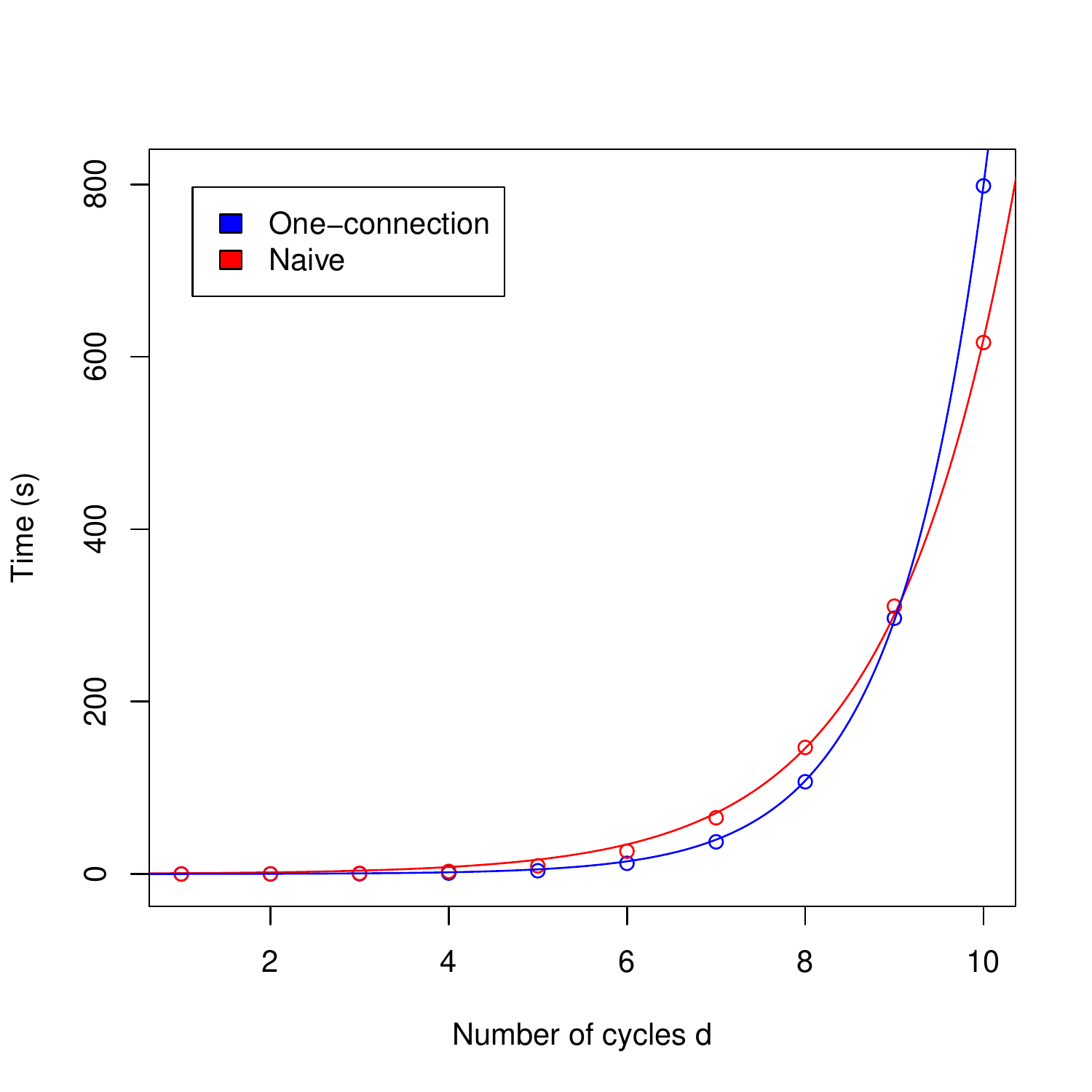}
    \caption{$2\ell=6$}
    \label{subfig:cycChain6}
  \end{subfigure}\\
  \begin{subfigure}[t]{.5\textwidth}
    \centering
    \includegraphics[width=\textwidth]{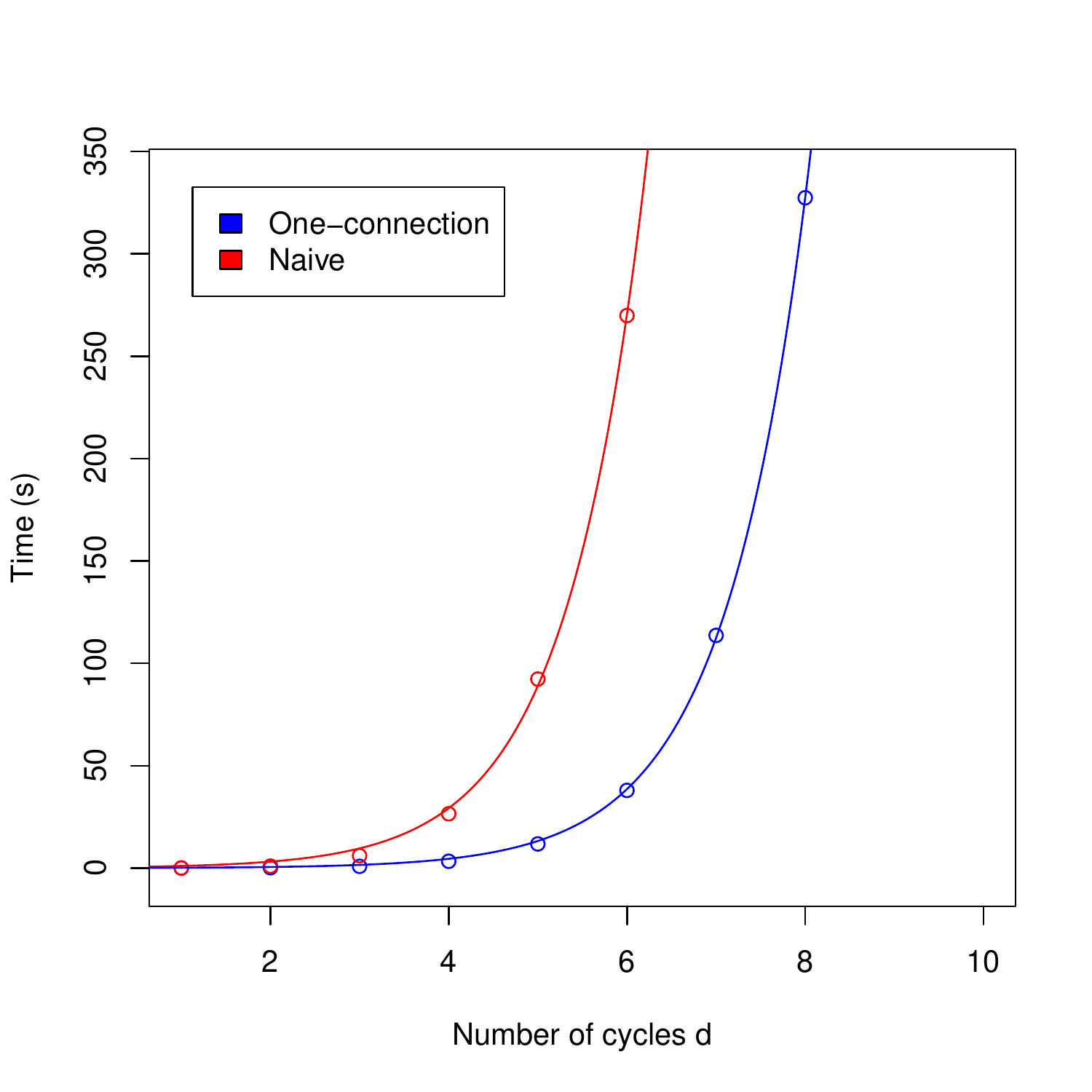}
    \caption{$2\ell=10$}
    \label{subfig:cycChain10}
  \end{subfigure}~
  \begin{subfigure}[t]{.5\textwidth}
    \centering
    \includegraphics[width=\textwidth]{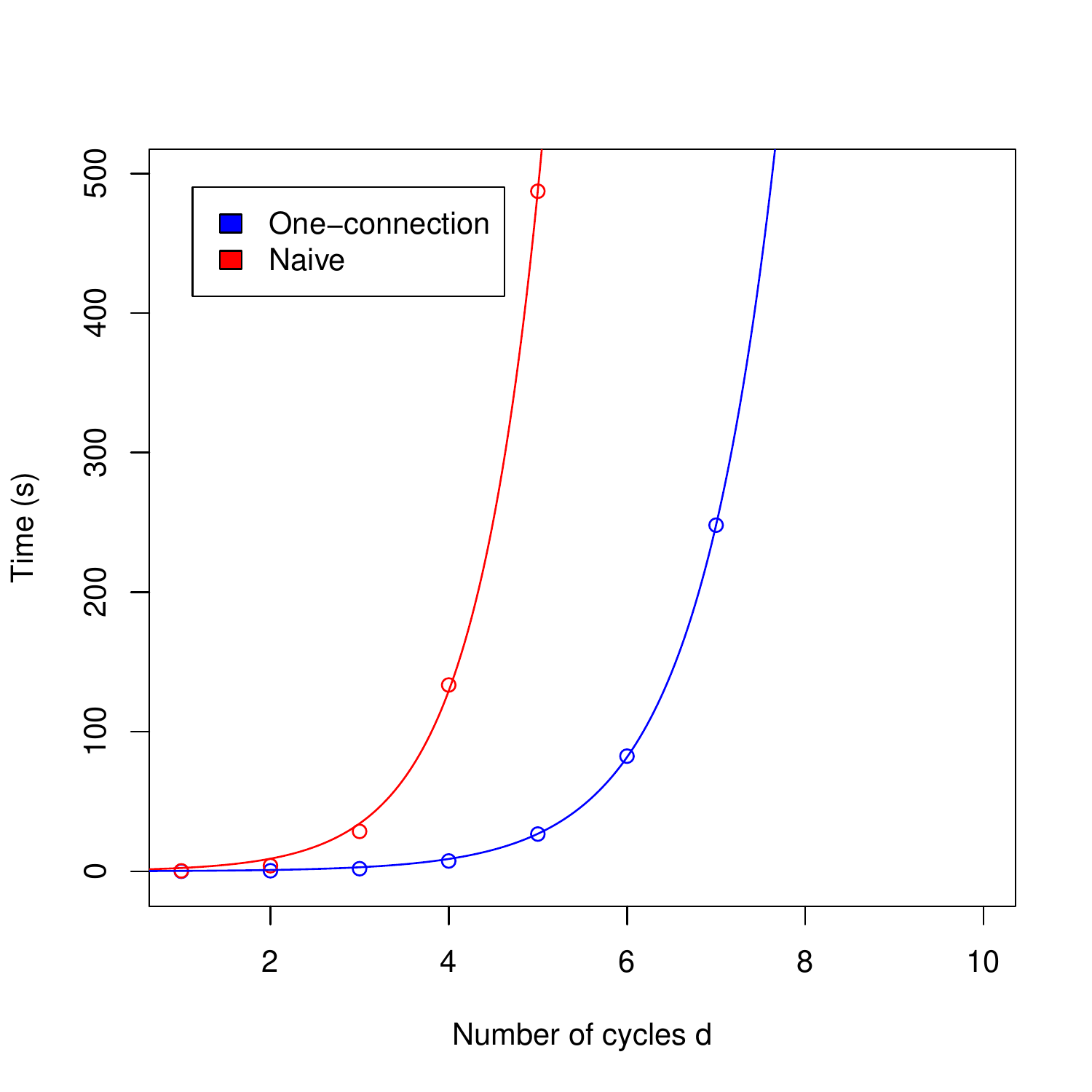}
    \caption{$2\ell=14$}
    \label{subfig:cycChain14}
  \end{subfigure}
  \caption{Timings for the computation of covariance matrices for chains of cycles.}
  \label{fig:cycleChainPlots}
\end{figure}

\subsubsection{Sparse random graphs} 
\label{ssub:sparse_random_graphs}
In the second computational experiment we consider random Erd\H os–-R\' enyi mixed graphs. We fix the vertices $\calV = [n]$, and for each ordered pair $(i,j)$, where $i,j \in [n]$, we add the directed edge $i \to j$ to $\calD$ with probability $p_D$. In addition, for each unordered pair $\{i,j\}$, where $i,j \in [n]$, we add a bidirected edge $i \leftrightarrow j$ to $\calB$ with probability $p_B$. We consider the following random graphs indexed by tuples $(n, p_D, p_B, c)$, where $c$ is the number of directed cycles in the mixed graph. We will choose the parameters
\begin{itemize}
  \item $n = 50, p_D = 0.020$
  \item $n = 100, p_D = 0.010$
  \item $n = 200, p_D = 0.005$,
\end{itemize}
and for each bullet, we let $p_B = 0, 0.01, 0.02, \dotsc, 0.1$, and $c = 0,1,\dotsc,10$. Thus we generate $121$ random graphs for each bullet for a total of $363$ graphs. We then compute the covariance matrix of each graph using the 1-connection method and the naive method, both with a time limit of 10 minutes.
We performed the experiments on a Mac Pro 
equipped with a 3,5 GHz 6-Core Intel Xeon processor  and 64GB RAM.

The 1-connection method was faster than the naive method in 85\%, 80\% and 79\% when $n=50, 100$, and $200$ respectively.

\begin{figure}
  \centering
  \begin{subfigure}[t]{.4\textwidth}
    \centering
    \includegraphics[width=\textwidth]{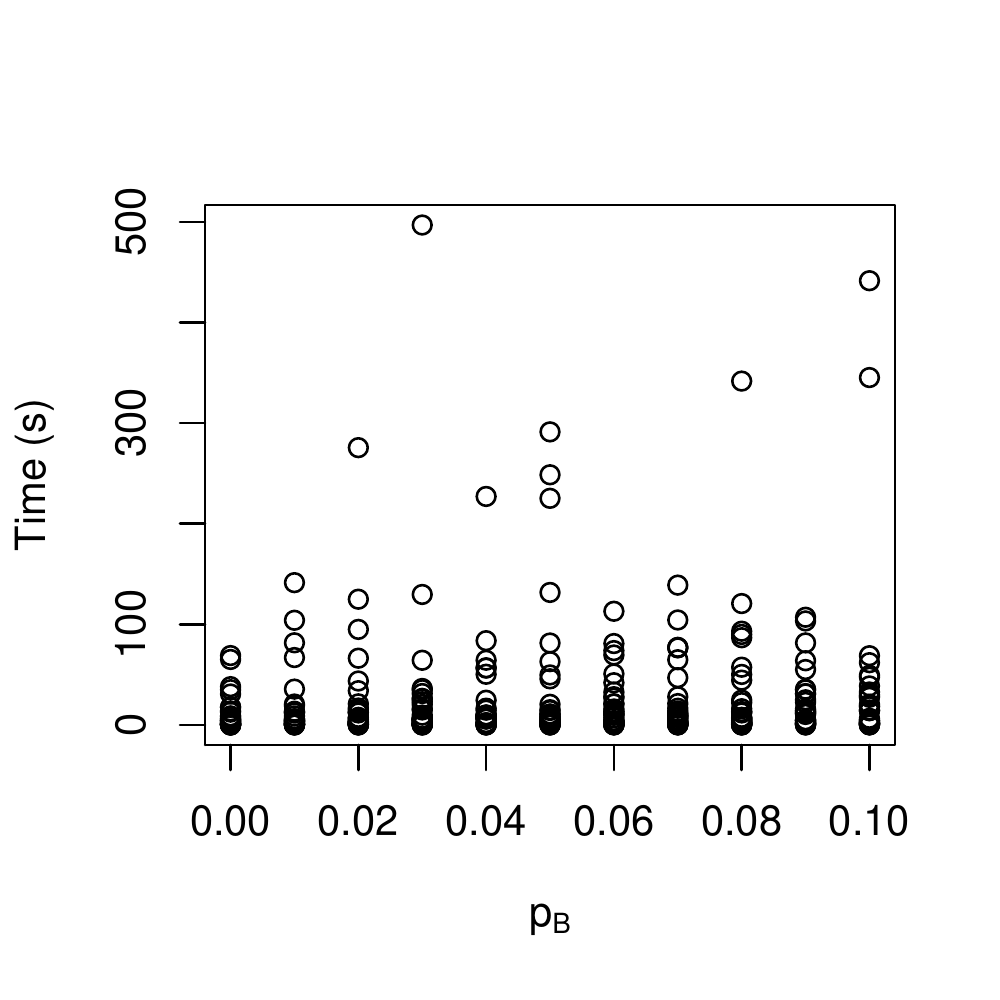}
    \caption{One-connection}
  \end{subfigure}~
  \begin{subfigure}[t]{.4\textwidth}
    \centering
    \includegraphics[width=\textwidth]{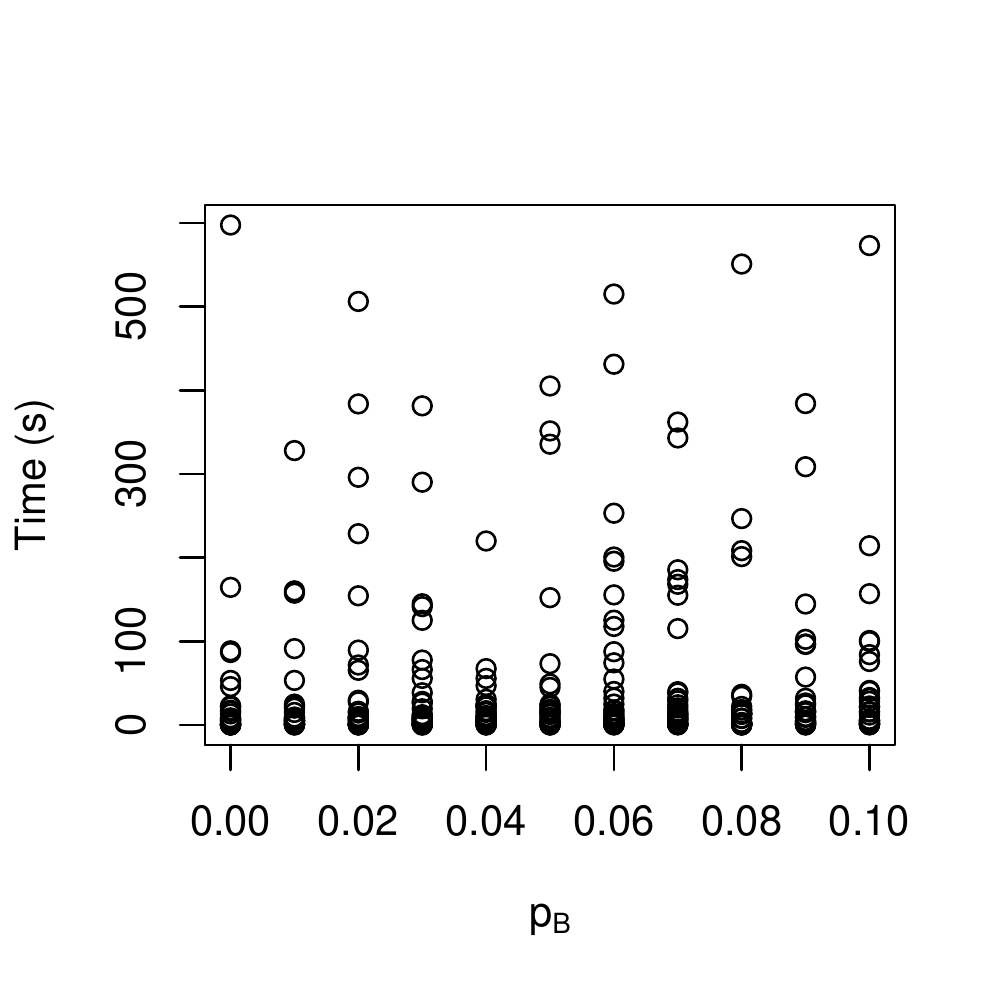}
    \caption{Naive method}
  \end{subfigure}
  \caption{Computation times for random sparse graphs with different bidirected edge probabilities $p_B$. Each circle corresponds to one graph.}
  \label{fig:pbvstime}
\end{figure}

\Cref{fig:pbvstime} shows the runtime for the random graphs with respect to the probability of adding a bidirected edge $p_B$. The graphs suggest that for both methods the number of bidirected edges doesn't matter much, indicating that the inversion of the matrix $(I-\Lambda)$ is the computational bottleneck.

For each pair $(n,c)$, where $n = 50,100,200$ is the number of vertices and $c = 0,1,\dotsc,10$ is the number of cycles, we have 11 graphs, one for each possible $p_B$. The mean runtimes of each are plotted in \Cref{fig:medianTimes}. As in the previous section, we observe that the 1-connection method will usually beat the naive method as the number of cycles increases.

\begin{figure}
  \centering
  \begin{subfigure}[t]{.5\textwidth}
    \centering
    \includegraphics[width=\textwidth]{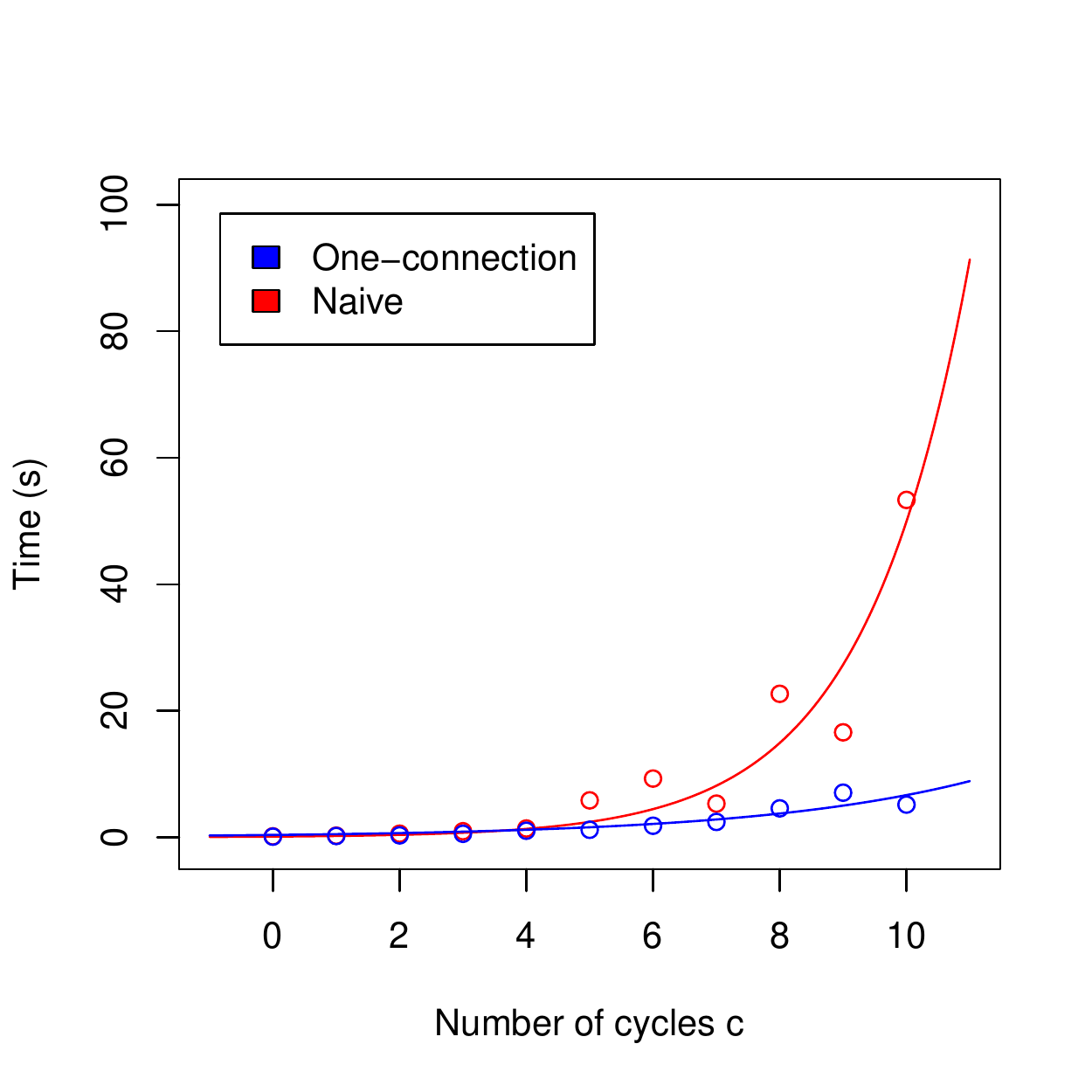}
    \caption{$n = 50$.}
  \end{subfigure}~
  \begin{subfigure}[t]{.5\textwidth}
    \centering
    \includegraphics[width=\textwidth]{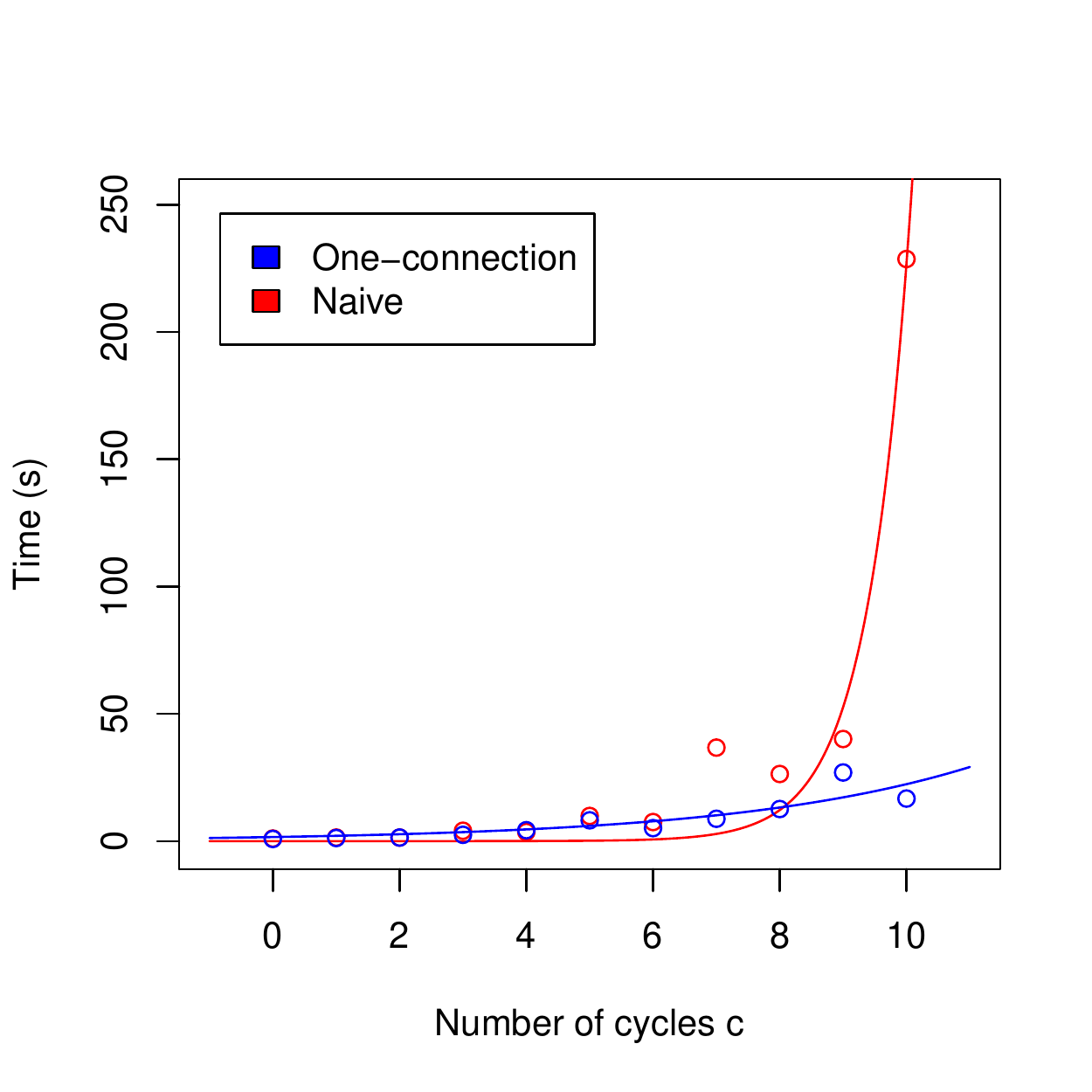}
    \caption{$n = 100$.}
  \end{subfigure}\\
  \begin{subfigure}[t]{.5\textwidth}
    \centering
    \includegraphics[width=\textwidth]{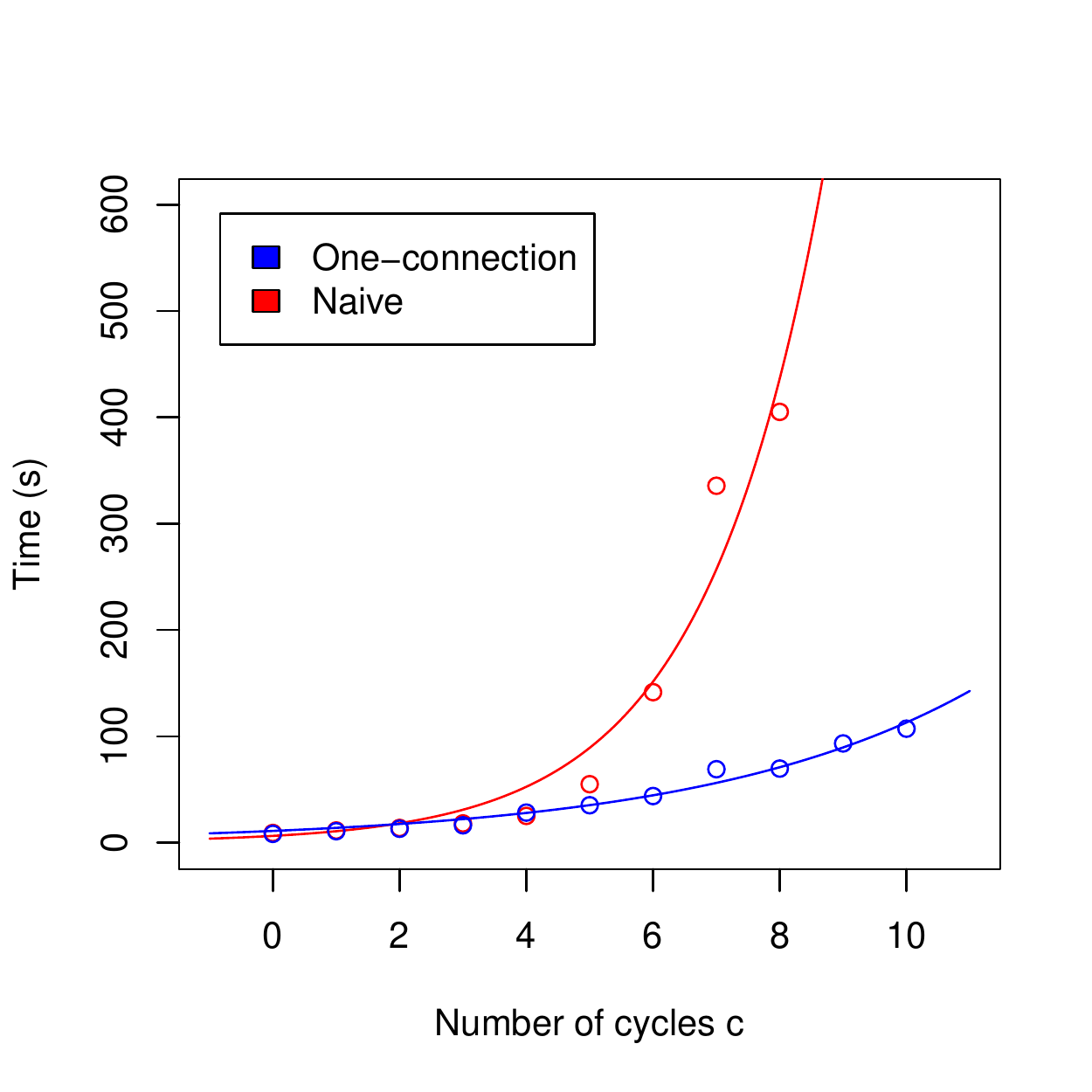}
    \caption{$n = 200$.}
  \end{subfigure}
  \caption{Median runtimes, with a fitted curve $t = a\cdot b^c$ for some $a,b$.}
  \label{fig:medianTimes}
\end{figure}

\section{Finding polynomials in the vanishing ideal}
\label{sec:test-in-ideal}

To determine whether two mixed graphs induce the same statistical model, we can investigate the polynomial relations between the entries of the covariance matrix. More precisely, let $G$ be a mixed graph, and recall that the statistical model $\calM_G$  consists of covariance matrices $\Sigma = \phi_G(\Lambda, \Omega)$. Instead of working with $\calM_G$, we will consider the \emph{(Gaussian) vanishing ideal} $\calI(G)$, defined by
\begin{equation}
  \calI(\calG) =
  \{f\in\R[\Sigma] \,|\, f(\Sigma)=0 \,\, \mbox{for all} \,\,
  \Sigma\in \calM_\calG\},
\end{equation}
where $\R[\Sigma]$ is the polynomial ring in $\sigma_{ij}$ with real coefficients. In other words, the vanishing ideal is the radical ideal corresponding to the Zariski closure of $\calM_G$.

If $G$ is a directed acyclic graph, the vanishing ideal can be computed by elimination
\begin{align*}
   \calI(G) = \langle \Sigma - \phi_G(\Lambda, \Omega) \rangle \cap \R[\Sigma].
\end{align*}
For general mixed graphs $G$, a common approach is to saturate with respect to $\det(I- \Lambda)$ and then perform elimination (see 
\cite{drton2018algebraic} for details).  Such a process uses Gr\"obner bases, and while relatively straightforward, the approach is
computationally prohibitive for graphs with even very few vertices. Thus a common aim is to develop combinatorial approaches that
exploit the structure of the underlying graph.
One example of the power of such an approach is described in \cite{geiger1990condindep}, where authors use the properties of the underlying graph to find conditional independence relations between the random variables, which in turn translate to the vanishing of almost principal minors of the covariance matrix.
Recently in \cite{drton2018nested},  authors introduced the notion
of nested determinants to obtain elements of the vanishing ideal of the model
that need not be given by determinantal constraints of submatrices of
$\Sigma.$ For example, in the case of the Verma graph in Figure \ref{fig:verma}, the authors found a
polynomial in the vanishing ideal that was not given by any subdeterminant of the covariance matrix, but could be represented as a nested determinant.

Our 1-connection framework gives us a concrete interpretation of the entries of the covariance matrix. Once symbolic expressions (in $\lambda$ and $\omega$-variables) for each entry of the covariance matrix are computed, testing whether or not some polynomial $f \in \R[\Sigma]$ is in $\calI(G)$ can be done by a simple substitution. This idea can also be exploited to find elements of $\calI(G)$ of small degree by creating a generic polynomial $f$ with indeterminate coefficients, substituting, and using linear algebra to solve for the coefficients. As the vanishing ideal is homogeneous, it suffices to consider homogeneous elements $f$. Naturally, this method will suffer from the combinatorial explosion of the number of terms in $f$, but this can be kept in control with the right heuristics to reduce the number of terms such as those illustrated in Example \ref{ex:findingHomogeneousPolynomials}.

\begin{figure}
\centering
\begin{tikzpicture}
\tikzset{vertex/.style = {shape=circle,draw,minimum size=0.5em}}
\tikzset{edge/.style = {->,> = latex'}}
\tikzset{bdedge/.style = {<->,> = latex'}}
\node[vertex] (a) at (0,0) {$1$};
\node[vertex] (b) at (2,0) {$2$};
\node[vertex] (c) at (4,0) {$3$};
\node[vertex] (d) at (6,0) {$4$};

\draw[edge] (a) to (b);
\draw[edge] (b) to (c);
\draw[edge] (c) to (d);
\draw[edge] (a) to[bend left] (c);
\draw[bdedge] (b) to[bend right] (d);

\end{tikzpicture}
\caption{The Verma graph}
\label{fig:verma}
\end{figure}
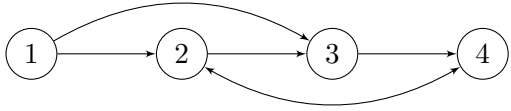

\begin{example}\label{ex:findingHomogeneousPolynomials}
  Let ${G}$ be the Verma graph, depicted in Figure \ref{fig:verma}. We will show that there are no degree 1 homogeneous polynomials in the vanishing ideal. Let $\Sigma = \phi_{G}(\Lambda,\Omega)$. Notice each  $\sigma_{ij}$, for $1\leq i\leq j \leq 4$, is linear in the $\omega$ variables. 
  Using the 1-connection method for example, we know exactly what each $\sigma_{ij}$ looks like. For example, we have
  \begin{align*}
    \sigma_{23} = (\lambda_{12}^2 \lambda_{23} + \lambda_{12}\lambda_{13}) \omega_{11} + \lambda_{23}\omega_{22} 
  \end{align*}

  We immediately observe that $\sigma_{11}$ is the only  $\sigma_{ij}$ that has a term with the monomial $\omega_{11}$. Thus if there were a homogeneous polynomial $f = \sum_{i,j} c_{ij} \sigma_{ij} \in \mathbb{R}[\Sigma]$ of degree 1  in the Gaussian vanishing ideal, we would have to have $c_{11} = 0$, otherwise there would be no way to cancel the $\omega_{11}$-term. 

  Amongst the $\sigma_{ij}$ still remaining, $\sigma_{12}$ is the only one containing a monomial $\lambda_{12}\omega_{11}$, so by the same reasoning as above, the monomial $\sigma_{12}$ can be disregarded, as it will not appear in any homogeneous degree 1 polynomial in the Gaussian vanishing ideal.

  We can repeat the procedure above, each time identifying a monomial unique to some $\sigma_{ij}$. In this particular example, in the end all $\sigma_{ij}$ will get eliminated, so we conclude there are no homogeneous degree 1 polynomials in the Gaussian vanishing ideal.

  For homogeneous degree 2 polynomials, we can use a similar procedure. Any degree 2 $\sigma$-monomial will be linear in the degree 2 $\omega$-monomials with coefficients being $\lambda$-polynomials, i.e.
  \begin{align*}
    \sigma^\alpha = \sum_{|\beta| = 2} k_\beta \omega^\beta,
  \end{align*}
  where $k_\beta \in \mathbb{R}[\Lambda]$ where $\R[\Lambda]$ is the ring of polynomials in $\lambda_{ij}$ with
real coefficients. 

  We observe that for example $\sigma_{1,4}\sigma_{4,4}$ is the only one among degree 2 monomials in $\sigma$ that has a the monomial $\lambda_{1,3}^3 \lambda_{3,4}^3 \omega_{1,1}^2$ in its support. Hence $\sigma_{1,4}\sigma_{4,4}$ cannot appear in any degree 2 homogeneous polynomial in the Gaussian vanishing ideal. We can repeat this procedure, each time eliminating a $\sigma$-monomial containing a unique monomial among the ones remaining. In this particular example, all degree 2 monomials get eliminated in the end, which shows that there are indeed no homogeneous degree 2 elements in the Gaussian vanishing ideal.
\end{example}

The procedure above gives rise to an algorithm that gives candidates for the monomial support of homogeneous polynomials in the Gaussian vanishing ideal. More precisely, for degree $d$ we consider the polynomials of the form
\begin{align*}
  f = \sum_{|\alpha| = d} c_\alpha \sigma^\alpha.
\end{align*}
The algorithm will return a subset $L$ of $\sigma$-monomials of degree $d$ such that any $f$ that lies in the Gaussian vanishing ideal will have $c_\alpha = 0$ for all $\sigma^\alpha \not\in L$. Thus the set $L$ contains the monomial support of degree $d$ homogeneous polynomials in the Gaussian vanishing ideal. In particular if $L$ is empty, there are no degree $d$ homogeneous polynomials in the Gaussian vanishing ideal. The full algorithm is shown in \Cref{alg:pruneSupport}.

\Cref{alg:pruneSupport} constructs a matrix where each entry is a set of $\lambda$-monomials. If the degree is large, constructing the matrix might take too long. \Cref{alg:weakPruneSupport} is a weakening of \Cref{alg:pruneSupport} which instead of considering a set of $\lambda$-monomials, we keep only the leading monomial with respect to some monomial order. This heuristic will in general be faster than \Cref{alg:pruneSupport}, but it may return a larger set of monomials. Nevertheless, \Cref{alg:weakPruneSupport} can serve as a quick first pass before running the more thorough \Cref{alg:pruneSupport}. Finally, one can obtain elements of the Gaussian vanishing ideal by computing syzygies over the base field of the monomials output by Algorithm \ref{alg:pruneSupport}.

\begin{algorithm}
    \caption{Compute monomial support of homogeneous elements in the Gaussian vanishing ideal}
    \label{alg:pruneSupport}
    \begin{algorithmic}[1]
        \Require A mixed graph $\calG = (\calV,\calD,\calB)$, a degree $d$.
        \Ensure A subset of $L$ of $\sigma$-monomials of degree $d$
        \Statex
        \Procedure{MonSupport}{$\calG$}
            \State $c \gets $ list of degree $d$ monomials in $\sigma$
            \State $r \gets $ list of degree $d$ monomials in $\omega$
            \State $M \gets $ adjugate matrix of $I-\Lambda$, see \eqref{eq:oneconnmatrix}
            \State $\varphi \gets $ the ring map $\mathbb{Q}[\Sigma] \to \mathbb{Q}[\Lambda,\Omega]$ taking $\sigma_{i,j}$ to $M_{i,j}$.
            \State $T \gets $ an $|r|\times |c|$ matrix with entries as follows: the entry corresponding to row $\omega^\alpha$ and column $\sigma^\beta$ is the set of $\lambda$-monomials in the coefficient of $\omega^\alpha$ in $\varphi(\sigma^\beta)$.
            \Repeat
              \State $\lambda^{\gamma'} \gets $ a $\lambda$-monomial occurring exactly once in its row
              \State $\sigma^{\beta'} \gets $ column where $\lambda^{\gamma'}$ appears.
              \State $T \gets $ $T$ with column $\sigma^{\beta'}$ removed
            \Until{No column is removed from $T$}
            \State \Return The set $L$ of monomials $\sigma^{\beta}$ corresponding to the remaining columns of $T$.
        \EndProcedure
    \end{algorithmic}
\end{algorithm}

\begin{algorithm}
    \caption{Compute a superset of the monomial support of homogeneous elements in the Gaussian vanishing ideal}
    \label{alg:weakPruneSupport}
    \begin{algorithmic}[1]
        \Require A mixed graph $\calG = (\calV,\calD,\calB)$, a degree $d$.
        \Ensure A subset of $L$ of $\sigma$-monomials of degree $d$
        \Statex
        \Procedure{WeakMonSupport}{$\calG$}
            \State $c \gets $ list of degree $d$ monomials in $\sigma$
            \State $r \gets $ list of degree $d$ monomials in $\omega$
            \State $M \gets $ adjugate matrix of $I-\Lambda$, see \eqref{eq:oneconnmatrix}
            \State $\varphi \gets $ the ring map $\mathbb{Q}[\Sigma] \to \mathbb{Q}[\Lambda,\Omega]$ taking $\sigma_{i,j}$ to $M_{i,j}$.
            \State $T \gets $ an $|r|\times |c|$ matrix with entries as follows: the entry corresponding to row $\omega^\alpha$ and column $\sigma^\beta$ is the leading $\lambda$-monomial in the coefficient of $\omega^\alpha$ in $\varphi(\sigma^\beta)$.
            \Repeat
              \State $\lambda^{\gamma'} \gets $ a $\lambda$-monomial occurring exactly once in its row
              \State $\sigma^{\beta'} \gets $ column where $\lambda^{\gamma'}$ appears.
              \State $T \gets $ $T$ with column $\sigma^{\beta'}$ removed
            \Until{No column is removed from $T$}
            \State \Return The set $L$ of monomials $\sigma^{\beta}$ corresponding to the remaining columns of $T$.
        \EndProcedure
    \end{algorithmic}
\end{algorithm}

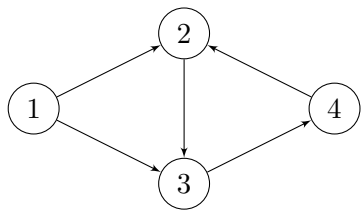
\begin{figure}
\centering
\begin{tikzpicture}
\tikzset{vertex/.style = {shape=circle,draw,minimum size=0.5em}}
\tikzset{edge/.style = {->,> = latex'}}
\node[vertex] (a) at (-2,0) {$1$};
\node[vertex] (b) at (0,1) {$2$};
\node[vertex] (c) at (0,-1) {$3$};
\node[vertex] (d) at (2,0) {$4$};

\draw[edge] (a) to (b);
\draw[edge] (a) to (c);
\draw[edge] (b) to (c);
\draw[edge] (c) to (d);
\draw[edge] (d) to (b);

\end{tikzpicture}
\caption{Cyclic graph in Example \ref{ex:drtonCyclic}.}
\label{fig:drtonCyclic}
\end{figure}

\begin{example}\label{ex:drtonCyclic}
  Consider the graph Figure \ref{fig:drtonCyclic}. We can use the 1-connection method to compute the adjugate matrix of $I-\Lambda$. Then, using the procedure above, we see that there are no homogeneous polynomials of degree $1,2,3,4$ and $5$ in the Gaussian vanishing ideal. In degree 6, we start with the set of all $\sigma$-monomials of degree 6, of which there are 5005. We run the weaker \Cref{alg:weakPruneSupport} to reduce this number to 3629. Then we can run \Cref{alg:pruneSupport} to further reduce the number to 31. A quick syzygy computation show exactly one relation among those 31 monomials:
  \begin{gather*}
    {\sigma}_{0 2}{\sigma}_{03}^{3}{\sigma}_{12}^{2}-2{\sigma}_{02}^{2}{\sigma}_{03}^{2}{\sigma}_{12}{\sigma}_{13}+{\sigma}_{02}^{3}{\sigma}_{03}{\sigma}_{13}^{2}-{\sigma}_{01}{\sigma}_{03}^{3}{\sigma}_{12}{\sigma}_{22}+{\sigma}_{01}{\sigma}_{02}{\sigma}_{03}^{2}{\sigma}_{13}{\sigma}_{22}\\
    +{\sigma}_{00}{\sigma}_{03}^{2}{\sigma}_{12}{\sigma}_{13}{\sigma}_{22}-{\sigma}_{00}{\sigma}_{02}{\sigma}_{03}{\sigma}_{13}^{2}{\sigma}_{22}+{\sigma}_{01}^{2}{\sigma}_{03}^{2}{\sigma}_{22}{\sigma}_{23}-{\sigma}_{00}{\sigma}_{03}^{2}{\sigma}_{11}{\sigma}_{22}{\sigma}_{23}-{\sigma}_{01}^{2}{\sigma}_{02}{\sigma}_{03}{\sigma}_{23}^{2}\\
    +{\sigma}_{00}{\sigma}_{02}{\sigma}_{03}{\sigma}_{11}{\sigma}_{23}^{2}+{\sigma}_{01}{\sigma}_{02}^{2}{\sigma}_{03}{\sigma}_{12}{\sigma}_{33}-{\sigma}_{00}{\sigma}_{02}{\sigma}_{03}{\sigma}_{12}^{2}{\sigma}_{33}-{\sigma}_{01}{\sigma}_{02}^{3}{\sigma}_{13}{\sigma}_{33}\\
    +{\sigma}_{00}{\sigma}_{02}^{2}{\sigma}_{12}{\sigma}_{13}{\sigma}_{33}-{\sigma}_{01}^{2}{\sigma}_{02}{\sigma}_{03}{\sigma}_{22}{\sigma}_{33}+{\sigma}_{00}{\sigma}_{02}{\sigma}_{03}{\sigma}_{11}{\sigma}_{22}{\sigma}_{33}+{\sigma}_{01}^{2}{\sigma}_{02}^{2}{\sigma}_{23}{\sigma}_{33}-{\sigma}_{00}{\sigma}_{02}^{2}{\sigma}_{11}{\sigma}_{23}{\sigma}_{33}.
  \end{gather*}
  Our result agrees with \cite[Example 3.6]{drton2009likelihood} and took about 7 minutes to compute. In our experiments with \textsc{Macaulay2} \cite{M2} and \textsc{Singular} \cite{singular}, computing elements in the Gaussian vanishing ideal via elimination as in \cite{drton2009likelihood} did not terminate after 12 hours. 
\end{example}

\section{Acknowledgements}
\label{sec:ack}

This material is based upon work supported by the National Science Foundation under Grant No. DMS-1439786 while the authors were in residence at the Institute for Computational and Experimental Research in Mathematics (ICERM) in Providence, RI, during the
semester on Nonlinear Algebra in Fall 2018.  This work is the product of a working group on graphical models that was held at ICERM during that semester.  Other members of the working group that contributed to initial explorations and discussions include: Alexandros Grosdos, Cvetelina Hill, Sara Lamboglia, Samantha Sherman, and Dane Wilburne.  We also would like to thank Elina Robeva for introducing us to structural equation models and sharing her expertise while at ICERM. 

BA gratefully acknowledges support through the Simons Institute for the Theory of Computing, University of California Berkeley, USA. EG is supported by National Science Foundation DMS-1945584. 
MH is partially supported by the Vilho, Yrj\"o and Kalle V\"ais\"al\"a Foundation and the Chateaubriand Fellowship. ET is partially supported by the ANR JCJC GALOP (ANR-17-CE40-0009), the
PGMO grant ALMA, and the PHC GRAPE.

\bibliographystyle{abbrv}
\bibliography{alg_stats}

\end{document}